\documentclass[11pt]{article}

\usepackage[margin=1in]{geometry}
\usepackage{bm}
\usepackage{amsmath,amssymb,amsfonts,amsthm}
\usepackage{mathtools}
\usepackage{booktabs}
\usepackage{multirow}
\usepackage{graphicx}
\usepackage{caption}
\usepackage{subfig}
\usepackage{float}
\usepackage[hidelinks]{hyperref}

\newcommand{\jmp}[1]{[\![#1]\!]}
\newcommand{\avg}[1]{\{\!\{#1\}\!\}}
\newcommand{\norm}[1]{\|#1\|}
\newcommand{\bnorm}[1]{\Big\|#1\Big\|}
\newcommand{\appendixnumberline}[1]{#1\quad}
\newcommand{\Acknowledgements}[1]{\section*{Acknowledgements}#1}

\theoremstyle{plain}
\newtheorem{theorem}{Theorem}[section]

\newtheorem{lemma}[theorem]{Lemma}
\newtheorem{proposition}[theorem]{Proposition}

\theoremstyle{definition}
\newtheorem{definition}[theorem]{Definition}

\theoremstyle{remark}
\newtheorem{remark}[theorem]{Remark}

\title{An Optimal IPDG Scheme for the Biharmonic Equation}

\author{%
Bohua Zhang\thanks{School of Mathematics and Statistics, Beijing Institute of
Technology, Beijing 100081, China (\texttt{bh\_zhang@bit.edu.cn}).}
\and
Xia Ji\thanks{School of Mathematics and Statistics, Beijing Institute of
Technology, Beijing 100081, China, and Beijing Key Laboratory on MCAACI,
Beijing Institute of Technology, Beijing 100081, China
(\texttt{jixia@bit.edu.cn}). Corresponding author.}
\and
Shuo Zhang\thanks{State Key Laboratory of Mathematical Sciences (SKLMS) and
State Key Laboratory of Scientific and Engineering Computing (LSEC), Institute
of Computational Mathematics and Scientific/Engineering Computing, Academy of
Mathematics and Systems Science, Chinese Academy of Sciences, Beijing 100190,
China, and School of Mathematical Sciences, University of Chinese Academy of
Sciences, Beijing 100049, China (\texttt{szhang@lsec.cc.ac.cn}).}%
}

\date{}

\begin{document}

\maketitle

\begin{abstract}
This paper presents an optimal interior penalty discontinuous Galerkin (IPDG)
scheme for the planar biharmonic equation using piecewise polynomials of degree
$k=3$ or $4$. In standard IPDG methods, large penalty parameters force the
discrete solution into an overconstrained space, severely degrading
accuracy---phenomenon known as numerical locking. To overcome this, our method
enforces only vertex continuity and projects the jumps of the function and its
normal derivative onto $\mathcal{P}^{k-3}$ and $\mathcal{P}^{k-2}$,
respectively. Consequently, as the penalty parameters tend to infinity, the
discrete solution is forced to lie in a constrained subspace $V_{h,\infty}^k$,
which we identify as the optimal nonconforming finite element space $B_h^k$.
This intrinsic connection fundamentally eliminates numerical locking. We prove
optimal error estimates of $\mathcal{O}(h^{k-1})$ in a mesh-dependent energy
norm and $\mathcal{O}(h^{k+1})$ in the $L^2$ norm. Numerical experiments on
both convex and L-shaped domains confirm that the proposed scheme is robust and
entirely locking-free, even for extremely large penalty parameters.
\end{abstract}



\medskip

\tableofcontents

\section{Introduction}
	\label{sec:introduction}

	Fourth-order elliptic partial differential equations, in particular the biharmonic problem $\Delta^2 u = f$, arise naturally in solid mechanics and fluid dynamics.
	Classical applications range from Kirchhoff-Love plate bending theory and the stream-function formulation of the incompressible Navier-Stokes equations to the Cahn-Hilliard phase-field model~\cite{Dong2021}.
	Standard Galerkin finite element approximations for such problems require $H^2$-conforming discrete spaces.
	However, constructing globally $C^1$-continuous basis functions on general triangulations is highly complex, typically demanding polynomials of at least degree $k \ge 5$ alongside intricate degrees of freedom.
	Since the theoretical energy-norm convergence rate for $H^m$ problems discretized with degree-$k$ piecewise polynomials is bounded by $\mathcal{O}(h^{k+1-m})$~\cite{lin2013}, achieving this upper bound with lower-degree polynomials (e.g., $k \in \{3,4\}$) is of significant practical interest to simplify element structures and reduce computational overhead.
	
	To bypass the strict topological constraints of conforming low-degree elements, discontinuous Galerkin (DG) methods have been widely adopted.
	Nitsche-type interior penalty (IP) methods~\cite{arnold1982,baker1977,Douglas1976,wheeler1978} provide a flexible alternative by relaxing global continuity requirements for elliptic problems~\cite{arnold2002}.
	For fourth-order equations specifically, IPDG schemes successfully achieve optimal $h$-convergence for arbitrary polynomial degrees~\cite{baker1977,brenner2005,Engel2002,Georgoulis2009,Hansbo2002,Mozolevski2007}.
	Despite these advantages, standard IPDG formulations require a sufficiently large penalty parameter to ensure discrete coercivity.
	An overly large penalty, however, frequently causes severe accuracy degradation---well-documented phenomenon known as numerical locking~\cite{Hansbo2002,Hansbo2011,JI2014}, where the error \emph{increases} dramatically when the penalty parameter exceeds a certain threshold.
	This occurs because, when the penalty parameter becomes too large, the discrete solution is forced into an overconstrained continuous space, destroying the scheme's approximation capability.
	While locking can be avoided if the limit function space is optimal (as seen with Morley and Argyris elements~\cite{Hansbo2002}), the necessary condition---an optimal $C^1$ subspace within the piecewise polynomial space~\cite{Hansbo2015} --- is rarely satisfied for cubic or quartic polynomials on unstructured grids.
	
	Thus, designing optimal IPDG methods for low-degree polynomials that remain robust under large penalization is a critical open problem. 
	We focus on the practically relevant low-order cases $k=3,4$, for which $C^1$-conforming spaces are not optimal on general grids; 
	for \(k\ge 5\), the classical Argyris element already yields an optimal \(C^1\)-conforming space and locking is not a concern. Thus, our work closes the gap for the most practically relevant low-order cases.
	
	In this paper, we introduce a unified reduced-penalty IPDG framework that circumvents numerical locking for the planar biharmonic problem on general triangulations for $k\in\{3,4\}$.
	Our approach is built on a partially continuous finite element space (vertex continuity only) and an adaptive projection strategy:
	the jumps of the function and its normal derivative are penalized only after being projected onto $\mathcal{P}^{k-3}$ and $\mathcal{P}^{k-2}$, respectively.
	The central theoretical insight is that, as the penalty parameters tend to infinity, the DG solution converges to the solution of a nonconforming finite element method whose underlying space is exactly the optimal nonconforming space $B_h^k$ introduced by Zhang S.~\cite{Zhang2021,Zhang2018_preprint}.
	Because $B_h^k$ possesses the optimal $\mathcal{O}(h^{k-1})$ approximation property in the energy norm, our IPDG scheme remains optimally accurate \emph{uniformly} in the penalty magnitude---thus completely eliminating numerical locking for cubic and quartic elements.
	
	We prove optimal \textit{a priori} error estimates: $\mathcal{O}(h^{k-1})$ in a mesh-dependent energy norm and $\mathcal{O}(h^{k+1})$ in the $L^2$ norm for solutions in $H^{k+1}(\Omega)$.  Numerical experiments on both convex and non-convex domains validate the theoretical predictions and demonstrate locking-free behavior even for extremely large penalty parameters (up to $10^5$).
	
	The remainder of this paper is organized as follows. 
	Section~\ref{sec:preliminaries} introduces the model problem, notations, and the partially continuous finite element space.
	Section~\ref{sec:formulation} presents the discrete IPDG formulation with reduced-order penalty. Section~\ref{sec:analysis} establishes stability and optimal \textit{a priori} error estimates.
	Section~\ref{sec:numerical} provides numerical experiments, and Section~\ref{sec:conclusion} gives concluding remarks.
\section{IPDG schemes for biharmonic equation}
	\label{sec:preliminaries}
	
	We consider the classical fourth-order biharmonic equation with clamped boundary conditions as our model problem:
	\begin{equation}
		\begin{cases}
			\Delta^2 u = f & \text{in } \Omega, \\
			u = \partial_n u = 0 & \text{on } \partial\Omega,
		\end{cases}
	\end{equation}
	where $\Omega \subset \mathbb{R}^2$ is a polygonal domain, and the source term satisfies $f \in L^2(\Omega)$. The variational formulation is to find $u\in H^2_0(\Omega)$, such that 
	\begin{equation}
	\int_\Omega\nabla^2u:\nabla^2v=\int_\Omega fv,\quad\forall\,v\in H^2_0(\Omega). 
	\end{equation}

	\subsection{The Unified Discrete IPDG Formulation}
	\label{sec:formulation}

	Let $\mathcal{T}_h$ be a shape-regular triangulation of the domain $\Omega$.
	We denote the set of interior edges by $\mathcal{E}_h^i$ and boundary edges by $\mathcal{E}_h^b$, with $\mathcal{E}_h = \mathcal{E}_h^i \cup \mathcal{E}_h^b$.
	Let $e \in \mathcal{E}_h^i$ be an interior edge shared by two adjacent elements $T^+$ and $T^-$.
	We assign a fixed unit normal vector $\bm{n}_e$ to each edge $e$, pointing from $T^+$ to $T^-$.
	We adopt standard DG notation, defining the average $\avg{\cdot}$ and the scalar jump $\jmp{\cdot}$ operators across an interior edge $e$ as conventional:
	\begin{equation}
		\avg{v} = \frac{1}{2}(v^+ + v^-), \quad \jmp{v} = v^+ - v^-.
	\end{equation}
	The normal derivative evaluated along $\bm{n}_e$ is defined as $\partial_n v = \nabla v \cdot \bm{n}_e$.
	Consequently, the jump and average of the normal derivative are given by $\jmp{\partial_n v} = \partial_n v^+ - \partial_n v^-$ and $\avg{\partial_n v} = \frac{1}{2}(\partial_n v^+ + \partial_n v^-)$.
	On boundary edges $e \in \mathcal{E}_h^b$, the normal $\bm{n}_e$ coincides with the outward unit normal of $\Omega$.
	To respect the homogeneous clamped boundary conditions, analogous trace definitions are applied: $\avg{v} = v$ and $\jmp{v} = v$ (where the latter denotes the trace value), and similarly $\avg{\partial_n v} = \partial_n v$ and $\jmp{\partial_n v} = \partial_n v$.
	
	Throughout this paper, we utilize the standard trace inequality and inverse estimate for discrete polynomials $v \in \mathcal{P}^k(T)$.
	Specifically, for any element $T \in \mathcal{T}_h$, there exist positive constants $C_{tr}$ and $C_{inv}$ such that:
	\begin{align}
		\norm{v}_{0,\partial T}^2 &\leq C_{tr} \left( h_T^{-1} \norm{v}_{0,T}^2 + h_T \norm{\nabla v}_{0,T}^2 \right), \label{eq:trace_ineq} \\
		\norm{\nabla^m v}_{0,T} &\leq C_{inv} h_T^{-m} \norm{v}_{0,T}. \label{eq:inverse_ineq}
	\end{align}
	
	To avoid the numerical locking inherent in fully penalized high-order DG methods, we restrict our trial and test functions to a partially continuous finite element space.
	Let $\mathcal{V}_h$ denote the set of all vertices in the triangulation $\mathcal{T}_h$.
	For a given polynomial degree $k \in \{3, 4\}$, we define the vertex-continuous discrete space $V_h^k$ as follows:
	\begin{equation}
		V_h^k = \left\{ v \in L^2(\Omega) : v|_T \in \mathcal{P}^k(T) , \forall T \in \mathcal{T}_h, \text{ and } \jmp{v}(\boldsymbol{z}) = 0 , \forall \boldsymbol{z} \in \mathcal{V}_h \right\}. \label{eq:space_definition}
	\end{equation}
	For boundary vertices $\boldsymbol{z} \in \partial\Omega$, the constraint $\jmp{v}(\boldsymbol{z}) = 0$ intrinsically enforces the strong boundary condition $v(\boldsymbol{z}) = 0$.
	
	The unified discrete bilinear form $A_h(\cdot, \cdot)$ defined on $V_h^k \times V_h^k$ incorporates a reduced-order projected penalty to prevent artificial stiffness and enhance matrix sparsity.
	The formulation is given by:
	\begin{equation}
		A_h(w, v) = a_h(w, v) + b_h(w, v) + c_h(w, v) + d_h(w, v). \label{eq:bilinear_form}
	\end{equation}
	Here, the effective shear force is defined with respect to the fixed edge normal $\bm{n}_e$ and its corresponding tangent $\bm{t}_e$, denoted by $T_e(w) = \partial_n \Delta w + \partial_{ntt} w$, where $\partial_{ntt} w = \partial_t(\partial_{nt} w)$. On straight-edged meshes the edge curvature term vanishes, so this expression is exact.
	The individual terms composing the bilinear form are given by:
	\begin{equation}
		\begin{aligned}
			a_h(w, v) &:= \sum_{T \in \mathcal{T}_h} \int_T \nabla^2 w : \nabla^2 v \, dx, \\
			b_h(w, v) &:= \sum_{e \in \mathcal{E}_h} \int_e \left( \avg{T_e(w)}\jmp{v} + \avg{T_e(v)}\jmp{w} \right) ds, \\
			c_h(w, v) &:= -\sum_{e \in \mathcal{E}_h} \int_e \left( \avg{\partial_{nn}w}\jmp{\partial_n v} + \avg{\partial_{nn}v}\jmp{\partial_n w} \right) ds, \\
			d_h(w, v) &:= \sum_{e \in \mathcal{E}_h} \int_e \left( \frac{\gamma_1}{h_e^3} \mathcal{P}_e^{k-3}\jmp{w} \mathcal{P}_e^{k-3}\jmp{v} + \frac{\gamma_2}{h_e} \mathcal{P}_e^{k-2}\jmp{\partial_n w} \mathcal{P}_e^{k-2}\jmp{\partial_n v} \right) ds,
		\end{aligned} \label{eq:discrete_operators}
	\end{equation}
	where $\mathcal{P}_e^m$ is the standard $L^2$-orthogonal projection operator onto the polynomial space $\mathcal{P}^m(e)$.
	Parameters $\gamma_1$ and $\gamma_2$ represent the penalty stabilization constants, dictating the imposition strength of the continuity constraints.
	
	To facilitate the subsequent analysis, we define the projected mesh-dependent energy norm on the discrete space:
	\begin{equation}
		\norm{v}_h^2 := \sum_{T \in \mathcal{T}_h} \norm{\nabla^2 v}_{0,T}^2 + \sum_{e \in \mathcal{E}_h} \frac{1}{h_e} \bnorm{\mathcal{P}_e^{k-2}\jmp{\partial_n v}}_{0,e}^2 + \sum_{e \in \mathcal{E}_h} \frac{1}{h_e^3} \bnorm{\mathcal{P}_e^{k-3}\jmp{v}}_{0,e}^2.
	\end{equation}
	Furthermore, to evaluate continuous functions properly (which may lack polynomial trace bounds), we introduce the augmented space $V(h) = V_h^k + (H^4(\Omega) \cap H_0^2(\Omega))$.
	This augmented space is equipped with the extended norm:
	\begin{equation}
		\norm{v}_{h,*}^2 := \norm{v}_h^2 + \sum_{T \in \mathcal{T}_h} \left( h_T \norm{\partial_{nn} v}_{\partial T}^2 + h_T^3 \norm{T_e(v)}_{\partial T}^2 \right).
	\end{equation}
	
	\begin{remark}[Implementation of the Projection Operators]
		In practical matrix assembly, the $L^2$-orthogonal projections $\mathcal{P}_e^{k-3}$ and $\mathcal{P}_e^{k-2}$ are implemented by solving local mass matrix systems on each edge $e$.
		To avoid numerical integration errors that could pollute the optimal convergence rates, we utilize 1D Gauss-Legendre quadrature rules with exactly $k+1$ points for all edge integrals.
		The projection bases are chosen as the normalized 1D Legendre polynomials mapped onto each edge, which orthogonalizes the local mass matrix and significantly accelerates the element-level assembly process.
		Furthermore, because we use normalized 1D Legendre polynomials as the projection basis, the local mass matrix on each edge reduces to a diagonal matrix (entries are squares of norms of Legendre polynomials). Consequently, computing the projected jumps requires only a few dot products, and the overall cost of the penalty terms is comparable to that of a standard full-jump IPDG implementation.
	\end{remark}

	\subsection{Discrete Stability Analysis}
	\label{subsec:stability}

	\begin{lemma}[Consistency and Galerkin Orthogonality] \label{lem:orthogonality}
		Assume the exact solution $u \in H^s(\Omega)$ with $s \geq 4$.
		Then, the inter-element jumps vanish, i.e., $\jmp{u} = 0$ and $\jmp{\partial_n u} = 0$, yielding the Galerkin orthogonality: $A_h(u - u_h, v) = 0$ for all $v \in V_h^k$.
	\end{lemma}
	\begin{proof}
		For the exact solution $u \in H^4(\Omega)$, its trace and normal derivative are single-valued across any interior edge $e \in \mathcal{E}_h^i$.
		Thus, the jumps $\jmp{u}$ and $\jmp{\partial_n u}$ trivially vanish. Multiplying the biharmonic equation by a test function $v \in V_h^k$ and integrating by parts element-wise immediately yields the standard Galerkin orthogonality.
	\end{proof}

	\begin{lemma} \label{lem:norm_equiv}
		By standard discrete Poincar\'{e}-Friedrichs type inequalities for piecewise polynomials, the full jump along the edges is bounded globally by the broken Hessian.
		Consequently, the projected norm $\norm{\cdot}_h$ is uniformly equivalent to the full jump DG norm on the discrete space $V_h^k$.
	\end{lemma}
	\begin{proof}
		For any $v \in V_h^k$, the function value jump $\jmp{v}$ is a piecewise polynomial of degree $k$ along any edge $e \in \mathcal{E}_h$.
		Due to the topological constraint of the space defined in \eqref{eq:space_definition}, $\jmp{v}$ vanishes at all vertices $\boldsymbol{z} \in \mathcal{V}_h$.
		On any given edge $e$, the semi-norm $\norm{\partial_s^2 p}_{0,e}$ constitutes a true norm on the finite-dimensional subspace $\{ p \in \mathcal{P}^k(e) : p(\text{endpoints}) = 0 \}$.
		This holds because a polynomial with a vanishing second derivative is linear, and a linear function vanishing at both endpoints is identically zero.
		By the equivalence of norms on finite-dimensional spaces and a standard inverse estimate scaled to the edge $e$, the unpenalized high-frequency components are controlled by the local tangential derivatives:
		\begin{equation}
			h_e^{-3} \bnorm{(I - \mathcal{P}_e^{k-3})\jmp{v}}_{0,e}^2 \leq C h_e \norm{\partial_s^2 (\jmp{v})}_{0,e}^2,
		\end{equation}
		where $\partial_s$ denotes the tangential derivative along $e$.
		
		Summing this over all edges and applying the standard discrete inverse trace inequality (i.e., $\norm{\nabla^2 v}_{0,e}^2 \le C h_T^{-1} \norm{\nabla^2 v}_{0,T}^2$), the high-frequency components are globally bounded by the broken Hessian:
		\begin{equation}
			\sum_{e \in \mathcal{E}_h} h_e^{-3} \bnorm{(I - \mathcal{P}_e^{k-3})\jmp{v}}_{0,e}^2 \leq C \sum_{T \in \mathcal{T}_h} \norm{\nabla^2 v}_{0,T}^2. \label{eq:poincare_jump}
		\end{equation}
		Similarly, for the normal derivative jump $\jmp{\partial_n v} \in \mathcal{P}^{k-1}(e)$, its orthogonal complement to $\mathcal{P}^{k-2}(e)$ is bounded by its tangential variation:
		\begin{equation}
			\begin{aligned}
				\sum_{e \in \mathcal{E}_h} h_e^{-1} \bnorm{(I - \mathcal{P}_e^{k-2})\jmp{\partial_n v}}_{0,e}^2 &\leq C \sum_{e \in \mathcal{E}_h} h_e \norm{\partial_s (\jmp{\partial_n v})}_{0,e}^2 \\
				&\leq C \sum_{T \in \mathcal{T}_h} \norm{\nabla^2 v}_{0,T}^2. \label{eq:poincare_normal}
			\end{aligned}
		\end{equation}
		Since the total jump is the orthogonal sum of the projected lower-order components and the unpenalized higher-order components, bounds \eqref{eq:poincare_jump} and \eqref{eq:poincare_normal} demonstrate that the projected norm $\norm{v}_h$ provides equivalent control over the discrete functions as the standard full-jump DG norm.
		The equivalence constants depend only on the polynomial degree $k$ and the shape regularity of $\mathcal{T}_h$, strictly independent of $h$.
		
		The case $k=4$ follows completely analogous lines. For the function value jump $\jmp{v} \in \mathcal{P}^4(e)$ that vanishes at the endpoints, the semi-norm $\|\partial_s^2 p\|_{0,e}$ remains a norm on the finite-dimensional space $\{p \in \mathcal{P}^4(e) : p(\text{endpoints}) = 0\}$. By dimension counting, the boundary conditions and the orthogonal projection operator $(I - \mathcal{P}_e^1)$ collectively eliminate all degrees of freedom for the linear polynomial kernel of the semi-norm $\|\partial_s^2 \cdot\|_{0,e}$, thereby ensuring it is a proper norm on the orthogonal complement space.
		Hence, by equivalence of norms on this space together with a standard scaling argument, we obtain
		\begin{equation}
			h_e^{-3} \norm{(I - \mathcal{P}_e^1)\jmp{v}}_{0,e}^2 \leq C h_e \norm{\partial_s^2(\jmp{v})}_{0,e}^2.
		\end{equation}
		For the normal derivative jump $\jmp{\partial_n v} \in \mathcal{P}^3(e)$, we need to control the orthogonal complement of $\mathcal{P}^2(e)$.
		Any polynomial $q \in \mathcal{P}^3(e)$ with $\partial_s q = 0$ is constant. Since constant functions belong to $\mathcal{P}^2(e)$, orthogonality to $\mathcal{P}^2(e)$ forces the constant to vanish.
		Thus $\|\partial_s q\|_{0,e}$ is a norm on the complement space $\mathcal{P}^3(e) \ominus \mathcal{P}^2(e)$.
		By finite-dimensional norm equivalence and the usual scaling, we get
		\begin{equation}
			h_e^{-1} \norm{(I - \mathcal{P}_e^2)\jmp{\partial_n v}}_{0,e}^2 \leq C h_e \norm{\partial_s(\jmp{\partial_n v})}_{0,e}^2.
		\end{equation}
		Summing over all edges and applying the discrete trace inequality exactly as for $k=3$ gives the global bounds also for $k=4$.
		Therefore, Lemma~\ref{lem:norm_equiv} holds verbatim for $k \in \{3,4\}$.
	\end{proof}

	\begin{proposition}[Projection identities for the cross terms] \label{prop:orthogonality}
		For any discrete function $v \in V_h^k$, because $v|_T \in \mathcal{P}^k(T)$, the traces of $v$ inherently satisfy $\partial_{nn} v|_e \in \mathcal{P}^{k-2}(e)$ and $T_e(v)|_e \in \mathcal{P}^{k-3}(e)$.
		This polynomial degree matching yields the exact projection identities:
		\begin{align}
			\int_e \avg{\partial_{nn} v} \jmp{\partial_n v} ds &= \int_e \avg{\partial_{nn} v} \mathcal{P}_e^{k-2} \jmp{\partial_n v} ds, \label{eq:ortho_nn} \\
			\int_e \avg{T_e(v)} \jmp{v} ds &= \int_e \avg{T_e(v)} \mathcal{P}_e^{k-3} \jmp{v} ds. \label{eq:ortho_T}
		\end{align}
		Note that for $v \in V_h^k$, the trace $\partial_{nn}v|_e \in \mathcal{P}^{k-2}(e)$ already belongs to the target space of $\mathcal{P}_e^{k-2}$, so the orthogonality $\mathcal{P}_e^{k-2}(\partial_{nn}v) = \partial_{nn}v$ is trivially satisfied.
	\end{proposition}

	\begin{theorem}[Coercivity] \label{thm:stability}
		For sufficiently large penalty parameters $\gamma_1, \gamma_2 > 0$ strictly greater than threshold constants derived from trace and inverse inequalities, there exists a positive constant $\alpha$ such that:
		\begin{align}
			A_h(v, v) &\geq \alpha\|v\|_h^2, \quad \forall v \in V_h^k. \label{eq:coercivity}
		\end{align}
	\end{theorem}
	\begin{proof}
		Invoking the orthogonal projection properties established in Proposition \ref{prop:orthogonality}, the cross-terms simplify directly.
		Since $T_e(v)$ involves third-order derivatives of $v$, applying the discrete trace inequality and the standard inverse inequality yields $h_e^3 \norm{\avg{T_e(v)}}_{0,e}^2 \le C h_T^2 \norm{\nabla^3 v}_{0,T}^2 \le C' \norm{\nabla^2 v}_{0,T}^2$.
		We can bound the cross term $b_h(v,v)$ as follows:
		\begin{equation}
			\begin{aligned}
				|b_h(v,v)| &= \left| 2 \sum_{e \in \mathcal{E}_h} \int_e \avg{T_e(v)} \mathcal{P}_e^{k-3}\jmp{v} ds \right| \\
				&\leq 2 \left( \sum_{e \in \mathcal{E}_h} h_e^3 \bnorm{\avg{T_e(v)}}_{0,e}^2 \right)^{1/2} \left( \sum_{e \in \mathcal{E}_h} h_e^{-3} \bnorm{\mathcal{P}_e^{k-3}\jmp{v}}_{0,e}^2 \right)^{1/2} \\
				&\leq C_1 \norm{\nabla^2 v}_{0,\mathcal{T}_h} \left( \sum_{e \in \mathcal{E}_h} \frac{1}{h_e^3} \bnorm{\mathcal{P}_e^{k-3}\jmp{v}}_{0,e}^2 \right)^{1/2}.
			\end{aligned}
		\end{equation}
		Applying the standard Young's inequality with $\epsilon > 0$ yields:
		\begin{equation}
			|b_h(v,v)| \leq \epsilon \norm{\nabla^2 v}_{0,\mathcal{T}_h}^2 + \frac{C_1^2}{4\epsilon} \sum_{e \in \mathcal{E}_h} \frac{1}{h_e^3} \bnorm{\mathcal{P}_e^{k-3}\jmp{v}}_{0,e}^2.
		\end{equation}
		A similar argument applied to the normal derivative cross term $c_h(v,v)$ yields a constant $C_2$:
		\begin{equation}
			|c_h(v,v)| \leq \epsilon \norm{\nabla^2 v}_{0,\mathcal{T}_h}^2 + \frac{C_2^2}{4\epsilon} \sum_{e \in \mathcal{E}_h} \frac{1}{h_e} \bnorm{\mathcal{P}_e^{k-2}\jmp{\partial_n v}}_{0,e}^2.
		\end{equation}
		Note that the global constant scalings $C_1$ and $C_2$ depend solely on the polynomial degree $k$, the shape regularity of the mesh $\mathcal{T}_h$, the trace inequality constant $C_{tr}$, and the inverse inequality constant $C_{inv}$, entirely independent of the mesh size $h$.
		Therefore, the thresholds for the penalty parameters $\gamma_1$ and $\gamma_2$ remain uniformly bounded as $h$ tends to $0$.
		Combining these bounds into the unified bilinear form gives:
		\begin{align}
				A_h(v, v) &\geq \left(1 - 2\epsilon\right) \norm{\nabla^2 v}_{0,\mathcal{T}_h}^2 \nonumber \\
				&\quad + \left(\gamma_1 - \frac{C_1^2}{4\epsilon}\right) \sum_{e \in \mathcal{E}_h} \frac{1}{h_e^3} \bnorm{\mathcal{P}_e^{k-3}\jmp{v}}_{0,e}^2 + \left(\gamma_2 - \frac{C_2^2}{4\epsilon}\right) \sum_{e \in \mathcal{E}_h} \frac{1}{h_e} \bnorm{\mathcal{P}_e^{k-2}\jmp{\partial_n v}}_{0,e}^2. \label{eq:coercivity_estimate}
		\end{align}
	
		By choosing $\epsilon = 1/4$, and selecting penalty parameters such that $\gamma_1 > C_1^2$ and $\gamma_2 > C_2^2$, we obtain $A_h(v,v) \geq \alpha \norm{v}_h^2$ for a strictly positive constant $\alpha$.

		From the estimate~\eqref{eq:coercivity_estimate} with $\epsilon=1/4$ we can identify the coercivity constant explicitly. Using the norm equivalence established in Lemma~\ref{lem:norm_equiv}, there exists a constant $c_{\rm eq}>0$ depending only on $k$ and the shape-regularity of $\mathcal{T}_h$ such that
		\begin{equation}\label{eq:alpha_expression}
			\alpha\;=\;c_{\rm eq}\, \min\!\Bigl\{\,\frac12,\;\gamma_1-C_1^2,\;\gamma_2-C_2^2\Bigr\}.
		\end{equation}
		Thus, when $\gamma_1$ and $\gamma_2$ are just above the thresholds $C_1^2$ and $C_2^2$, the coercivity constant $\alpha$ can be arbitrarily small; conversely, as $\gamma_1,\gamma_2\to\infty$, $\alpha$ tends to the fixed value $c_{\rm eq}/2$, which is completely independent of the penalty parameters.
	\end{proof}

	With the discrete stability and Galerkin orthogonality established, we now turn to the characterization of the limit space and the derivation of optimal error estimates.

	\section{Uniformly robust convergence of the scheme}
	\label{sec:analysis}
	
	The central theoretical insight of this paper is that the reduced-order penalty does not merely stabilize the scheme, but also forces the discrete solution to lie asymptotically in an optimal nonconforming finite element space. In this section we explicitly characterize the constrained subspace $V_{h,\infty}^k$ (the null space of the penalty terms) and show its equivalence to the known space $B_h^k$.

	\subsection{Constrained subspaces for limit problems}
	
	\subsubsection{Characterization of the Constrained Subspace and the Discrete Stokes Complex}
	A key property of the reduced-order penalization is that it forces the discrete solution to satisfy the projected jump conditions more and more stringently as the penalty grows. In the limit where $\gamma_1,\gamma_2$ tend to infinity, the solution must belong to the null space of the penalty terms, which we denote by
	\begin{equation}\label{eq:Vinf}
		V_{h,\infty}^k := \bigl\{ v \in V_h^k \mid \int_e \jmp{v} p_e = 0 \; \forall p_e\in\mathcal{P}^{k-3}(e),\;
		\int_e \jmp{\partial_n v} q_e = 0 \; \forall q_e\in\mathcal{P}^{k-2}(e), \; \forall e\in\mathcal{E}_h \bigr\}.
	\end{equation}
	
	It is important to note that the constrained subspace $V_{h,\infty}^k$ defined in~\eqref{eq:Vinf} coincides exactly with the optimal nonconforming space $B_h^k$ introduced by Zhang S.~\cite{Zhang2021,Zhang2018_preprint}. As we shall rigorously prove later in Proposition~\ref{prop:limit_convergence}, the DG solution converges to the nonconforming solution in $B_h^k$, which is unaffected by over-penalization; therefore numerical locking is completely avoided.
	
	Let $u_{h,\infty}\in V_{h,\infty}^k$ be the unique solution of the nonconforming finite element method
	\begin{equation}\label{eq:limit_problem}
		a_h(u_{h,\infty}, v_h) = (f, v_h), \qquad \forall\, v_h\in V_{h,\infty}^k.
	\end{equation}
	The well-posedness of this problem follows from the definiteness of the norm (established in Lemma~\ref{lem:norm_equiv}) and the Lax-Milgram lemma.
	
	\subsubsection{Interpolation Operator}
	To establish the approximation capability for the quartic case ($k=4$), the exactness of the discrete Stokes complex is paramount:
	\begin{equation}\label{eq:stokes_complex}
		0 \longrightarrow B_{h0}^4 \xrightarrow{\nabla_h} \bm{G}_{h0}^3 \xrightarrow{\mathrm{rot}_h} \mathbb{P}_{h0}^2 \longrightarrow 0.
	\end{equation}
	
	The case $k=3$ based on the Fortin-Soulie pair has been deeply analyzed in established research, specifically the work of Zhang S.~\cite{Zhang2021,Zhang2018_preprint}. Building upon this rigorous mathematical foundation, our primary theoretical contribution here is extending this interpolation mechanism and the corresponding discrete Stokes complex exactness to the highly challenging quartic case ($k=4$). For $k=4$, the exactness relies on the Crouzeix-Falk pair $\bm{G}_{h0}^3 \times \mathbb{P}_{h0}^2$. Historically, the inf-sup stability of this pair on general unstructured grids was a long-standing conjecture. However, this conjecture has been definitively resolved by the recent breakthrough work of Carstensen and Sauter \cite{Carstensen2022}. They rigorously proved that the Crouzeix-Raviart triangular elements of any odd degree $p \ge 3$ (which corresponds to our velocity space $\bm{G}_{h0}^3$) are unconditionally inf-sup stable on any regular triangulation containing at least one interior vertex.
	
	\begin{remark}
		By the fundamental theory of finite element complexes, the unconditional inf-sup stability is mathematically equivalent to the surjectivity of the discrete rotation operator $\mathrm{rot}_h$. Consequently, a standard homological dimension-counting argument immediately establishes the exactness identity $\nabla_h B_{h0}^4 = \ker(\mathrm{rot}_h)$. Therefore, the full exactness of the discrete Stokes complex \eqref{eq:stokes_complex} for both $k=3$ and $k=4$ holds rigorously without requiring any supplementary macro-element assumptions.
	\end{remark}
	
	We use a global interpolation operator $\mathcal{I}_h^k$ that maps smooth functions into the constrained subspace $V_{h,\infty}^k$ by means of the discrete Stokes complex, using the Fortin--Soulie pair for $k=3$ and the Crouzeix--Falk pair for $k=4$.
	
	\begin{definition}[Interpolation operator $\mathcal{I}_h^k$]\label{def:interp}
		For $u\in H^{k+1}(\Omega)\cap H_0^2(\Omega)$, define its interpolant $\mathcal{I}_h^k u\in V_{h,\infty}^k$ as follows.
		\begin{enumerate}
			\item Set $\bm{\phi} = \nabla u$; then $\bm{\phi}\in H^k(\Omega)^2\cap H_0^1(\Omega)$ and $\operatorname{rot}\bm{\phi}=0$.
			\item Solve the auxiliary Stokes problem: find $(\bm{\phi}_h,p_h)\in \bm{G}_{h0}^{k-1}\times \mathbb{P}_{h0}^{k-2}$ such that
			\begin{equation}\label{eq:stokes_interp}
				\begin{cases}
					(\nabla_h\bm{\phi}_h,\nabla_h\bm{\psi}_h) + (\operatorname{rot}_h\bm{\psi}_h,p_h)
					= (-\Delta\bm{\phi},\bm{\psi}_h), & \forall \bm{\psi}_h\in \bm{G}_{h0}^{k-1},\\[2mm]
					(\operatorname{rot}_h\bm{\phi}_h,q_h) = 0, & \forall q_h\in \mathbb{P}_{h0}^{k-2}.
				\end{cases}
			\end{equation}
			\item Because $\operatorname{rot}_h\bm{\phi}_h = 0$ and the discrete Stokes complex is exact,
			there exists a unique $u_I\in V_{h,\infty}^k$ such that $\nabla_h u_I = \bm{\phi}_h$.
			Define $\mathcal{I}_h^k u := u_I$.
		\end{enumerate}
	\end{definition}

	\subsubsection{Optimal Quasi-Interpolation Error Estimates}
	Given the unconditional exactness of this Stokes complex, it guarantees the existence of the quasi-interpolant required in the following lemma.
	
	\begin{lemma} \label{lem:interpolation}
		Assuming the exact solution possesses sufficient regularity $u \in H^{k+1}(\Omega)$, the global quasi-interpolant $\mathcal{I}_h^k u \in V_{h,\infty}^k$, defined in Definition~\ref{def:interp}, satisfies the approximation error bound:
		\begin{equation}
			\norm{u - \mathcal{I}_h^k u}_{h,*} \leq C h^{k-1} |u|_{k+1, \Omega}. \label{eq:interp_bound}
		\end{equation}
	\end{lemma}
	\begin{proof}
		Due to the overdetermined topological constraints of the limit space $V_{h,\infty}^k$, we cannot employ a standard local nodal interpolant. Instead, the interpolation operator $\mathcal{I}_h^k$ constructed via the auxiliary discrete Stokes problem in Definition~\ref{def:interp} is exactly the one introduced and analyzed by Zhang~\cite{Zhang2021, Zhang2018_preprint}.
		
		For the cubic case ($k=3$), the unconditional exactness of the discrete Stokes complex is based on the classical Fortin-Soulie pair~\cite{Fortin1983}. Lemma 3.2 in~\cite{Zhang2021} guarantees the optimal global approximation estimates:
		\begin{equation}\label{eq:approx3}
			\norm{u - \mathcal{I}_h^3 u}_{s,\mathcal{T}_h} \le C h^{4-s} |u|_{4,\Omega}, \quad s \in \{0,1,2,3,4\}.
		\end{equation}
		
		For the quartic case ($k=4$), the exactness relies on the Crouzeix-Falk pair. As resolved by the recent work of Carstensen and Sauter~\cite{Carstensen2022}, this discrete Stokes complex is unconditionally stable on general regular triangulations. This unconditional exactness guarantees that the auxiliary problem~\eqref{eq:stokes_interp} is well-posed, and commuting the quasi-interpolation with the gradient operator (the same commutating diagram technique as~\cite{Zhang2021,Zhang2018_preprint,Zhang2019_preprint}, now validated for $k=4$ by the unconditional stability proved in~\cite{Carstensen2022}) yields the corresponding optimal bounds:
		
		\begin{equation}\label{eq:approx4}
			\norm{u - \mathcal{I}_h^4 u}_{s,\mathcal{T}_h} \le C h^{5-s} |u|_{5,\Omega}, \quad s \in \{0,1,2,3,4,5\}.
		\end{equation}
		
		With the optimal approximation bounds established for the standard broken Sobolev norms, it remains to bound the extended norm $\|u-\mathcal{I}_h^k u\|_{h,*}$ associated with the IPDG formulation. Let $v = u - \mathcal{I}_h^k u$. Since the exact solution satisfies $\jmp{u}=0$ and $\jmp{\partial_n u}=0$, and the interpolant belongs strictly to the constrained limit space $\mathcal{I}_h^k u \in V_{h,\infty}^k$ (meaning its projected jumps automatically vanish), the penalty terms in the extended norm identically equal zero. The norm then reduces to the element interiors and boundary traces:
		\begin{equation}
			\|v\|_{h,*}^2 = \sum_{T \in \mathcal{T}_h} \|\nabla^2 v\|_{0,T}^2 + \sum_{T \in \mathcal{T}_h} \left( h_T \|\partial_{nn} v\|_{\partial T}^2 + h_T^3 \|T_e(v)\|_{\partial T}^2 \right).
		\end{equation}
		Applying the standard trace inequality \eqref{eq:trace_ineq} and summing over all elements as in Hansbo and Larson~\cite{Hansbo2002}, we bound the boundary terms by the local Sobolev semi-norms to obtain:
		\begin{equation}
			\|v\|_{h,*}^2 \leq C \left( \|v\|_{2,\Omega}^2 + h^2 \|v\|_{3,\Omega}^2 + h^4 \|v\|_{4,\Omega}^2 \right).
		\end{equation}
		Finally, substituting the global approximation properties \eqref{eq:approx3} and \eqref{eq:approx4} into the above inequality immediately yields the desired optimal bound \eqref{eq:interp_bound}.
	\end{proof}

	The degree-matched reduced-order projections are essential for robustness against large penalty parameters. By contrast, in standard full-penalty methods the penalized solution is driven into an overconstrained continuous space, whereas the solution of our scheme approaches a function in the optimal nonconforming space $B_h^k$, which retains full approximation power.

	\begin{proposition}[Convergence to the limit problem]\label{prop:limit_convergence}
		As the penalty parameters $\gamma_1,\gamma_2\to\infty$, the discrete solution $u_h$ of \eqref{eq:bilinear_form} satisfies
		\[
		\|u_h - u_{h,\infty}\|_h \to 0.
		\]
	\end{proposition}
	
	\begin{proof}
		We fix the mesh $\mathcal{T}_h$ and work on the finite-dimensional space $V_h^k$.
		By Theorem~\ref{thm:stability}, when $\gamma_1>C_1^2$ and $\gamma_2>C_2^2$ the coercivity estimate \eqref{eq:coercivity} holds:
		$\alpha\|u_h\|_h^2 \le A_h(u_h,u_h) = (f,u_h)$.
		By the Cauchy-Schwarz inequality and the discrete Poincar\'{e}-Friedrichs inequality (or Lemma~\ref{lem:norm_equiv}), we have $(f,u_h) \le C\|f\|_{0,\Omega}\|u_h\|_h$. 
		Hence, $\|u_h\|_h\le C\|f\|_{0,\Omega}/\alpha$ uniformly with respect to the penalty parameters.
		Thus, $\{u_h\}$ is bounded in the finite-dimensional space $V_h^k$.
		
		To see that the projected jumps tend to zero, we employ the coercivity estimate \eqref{eq:coercivity_estimate} before taking $\alpha$ (by choosing $\epsilon = 1/4$):
		\begin{align*}
			A_h(u_h,u_h)
			&\ge \frac12\|\nabla^2 u_h\|_{0,\mathcal{T}_h}^2
			+ \left(\gamma_1-C_1^2\right)\sum_{e\in\mathcal{E}_h}\frac1{h_e^3}\|\mathcal{P}_e^{k-3}\jmp{u_h}\|_{0,e}^2 \\
			&\quad + \left(\gamma_2-C_2^2\right)\sum_{e\in\mathcal{E}_h}\frac1{h_e}\|\mathcal{P}_e^{k-2}\jmp{\partial_n u_h}\|_{0,e}^2 .
		\end{align*}
		Since the left-hand side equals $(f,u_h)$ and $\|u_h\|_h$ is bounded, the right-hand side is bounded as well.
		Taking the limit as $\gamma_1,\gamma_2\to\infty$ implies that
		\begin{equation}\label{eq:proj_vanish}
			\mathcal{P}_e^{k-3}\jmp{u_h} \to 0, \qquad
			\mathcal{P}_e^{k-2}\jmp{\partial_n u_h} \to 0
			\quad\text{in }L^2(e)\text{ for every }e\in\mathcal{E}_h.
		\end{equation}
		
		Because $V_h^k$ is finite-dimensional and $\{u_h\}$ is bounded in $\|\cdot\|_h$, there exists a subsequence (still denoted by $u_h$) that converges strongly to some $u^*\in V_h^k$ in $\|\cdot\|_h$.
		The vanishing of the projected jumps \eqref{eq:proj_vanish} implies
		$\mathcal{P}_e^{k-3}\jmp{u^*}=0$ and $\mathcal{P}_e^{k-2}\jmp{\partial_n u^*}=0$; therefore $u^*\in V_{h,\infty}^k$.
		
		Now, let $v_h\in V_{h,\infty}^k$ be an arbitrary test function.
		Because the projected jumps of $v_h$ vanish, the penalty term $d_h(u_h,v_h)$ is identically zero.
		Applying Proposition~\ref{prop:orthogonality} (specifically the projection identities \eqref{eq:ortho_nn} and \eqref{eq:ortho_T}) and the symmetry of the bilinear forms, the cross-terms simplify to
		\begin{align*}
			b_h(u_h,v_h) &= \sum_{e\in\mathcal{E}_h}\int_e \avg{T_e(v_h)}\,\mathcal{P}_e^{k-3}\jmp{u_h}\,ds,\\
			c_h(u_h,v_h) &= -\sum_{e\in\mathcal{E}_h}\int_e \avg{\partial_{nn}v_h}\,\mathcal{P}_e^{k-2}\jmp{\partial_n u_h}\,ds.
		\end{align*}
		By \eqref{eq:proj_vanish}, both $b_h$ and $c_h$ tend to zero as $\gamma_1,\gamma_2\to\infty$.
		Passing to the limit in the original discrete formulation $A_h(u_h,v_h)=(f,v_h)$ yields
		\[
		a_h(u^*,v_h) = (f,v_h) \qquad \forall\, v_h\in V_{h,\infty}^k.
		\]
		Thus $u^*$ solves the limit problem \eqref{eq:limit_problem}. Uniqueness of the solution guarantees $u^*=u_{h,\infty}$.
		
		Since every convergent subsequence has the same limit $u_{h,\infty}$ and the whole sequence is bounded in a finite-dimensional space, the full sequence converges strongly in $\|\cdot\|_h$ to $u_{h,\infty}$.
	\end{proof}

\begin{remark}\label{rem:penalty_error}
	From the coercivity estimate~\eqref{eq:coercivity_estimate} with $\epsilon=1/4$ and the boundedness of $\|u_h\|_h$ (proof of Proposition~\ref{prop:limit_convergence}), one obtains
	\[
	\sum_{e\in\mathcal{E}_h}\frac1{h_e^3}\bnorm{\mathcal{P}_e^{k-3}\jmp{u_h}}_{0,e}^2
	= \mathcal O(\gamma_1^{-1}),\qquad
	\sum_{e\in\mathcal{E}_h}\frac1{h_e}\bnorm{\mathcal{P}_e^{k-2}\jmp{\partial_n u_h}}_{0,e}^2
	= \mathcal O(\gamma_2^{-1}),
	\]
	and consequently, the distance to the limit nonconforming solution is strictly bounded by:
	\[
	\|u_h - u_{h,\infty}\|_h \le C\bigl(\gamma_1^{-1/2}+\gamma_2^{-1/2}\bigr)\|f\|_{0,\Omega},
	\]
	where $C$ is strictly independent of $\gamma_1,\gamma_2$ and $h$. 
\end{remark}	

\subsection{Penalty-robust \textit{A Priori} Error Estimates}

		A central feature of the present scheme is that the penalty terms for the
		interpolation error vanish identically, which makes the continuity constant
		in the error analysis completely independent of the penalty parameters.
		Indeed, let $\eta = u - \mathcal{I}_h^k u$.  Since the exact solution satisfies
		$\jmp{u}=0$ and $\jmp{\partial_n u}=0$, and the interpolant belongs to
		$V_{h,\infty}^k$ by Definition~\ref{def:interp} (so that
		$\mathcal{P}_e^{k-3}\jmp{\mathcal{I}_h^k u}=0$ and
		$\mathcal{P}_e^{k-2}\jmp{\partial_n \mathcal{I}_h^k u}=0$), the projected jumps
		of $\eta$ vanish.  Consequently the entire penalty term is identically zero:
		\begin{equation}
			d_h(\eta, \xi) \equiv 0, \qquad \forall \xi \in V_h^k.
		\end{equation}
		This observation is the key to establishing penalty-robust continuity.

		\begin{lemma}[Mixed continuity]\label{lem:mixed_cont}
			For any $\eta \in V(h)$ satisfying
			$\mathcal{P}_e^{k-3}\jmp{\eta}=0$ and $\mathcal{P}_e^{k-2}\jmp{\partial_n\eta}=0$
			on every $e\in\mathcal{E}_h$, the bilinear form admits the bound
			\begin{equation}
				|A_h(\eta, \xi)| \leq M\|\eta\|_{h,*}\|\xi\|_h, \quad \forall \xi \in V_h^k,
			\end{equation}
			where the constant $M$ depends only on the domain and the shape regularity
			of the mesh, and is \textbf{strictly independent of the penalty parameters}
			$\gamma_1$ and $\gamma_2$.
		\end{lemma}
		\begin{proof}
			Under the hypotheses on $\eta$, the penalty term vanishes:
			$d_h(\eta,\xi)=0$.  The volume term $a_h(\eta, \xi)$ is bounded by the
			Cauchy--Schwarz inequality.  For the cross terms $b_h(\eta, \xi)$ and
			$c_h(\eta, \xi)$, the full unprojected jumps of $\xi$ are controlled
			globally by $\|\xi\|_h$ via Lemma~\ref{lem:norm_equiv}, while the
			continuous traces of $\eta$ are bounded by $\|\eta\|_{h,*}$ by definition
			of the extended norm.  Hence the stated estimate holds with $M$ independent
			of $\gamma_1,\gamma_2$.
		\end{proof}

		\begin{theorem}[Energy Norm Error Estimate] \label{thm:energy_error}
		Assume that the penalty parameters $\gamma_1,\gamma_2$ are chosen large enough so that the coercivity \eqref{eq:coercivity} holds.
		Let $u \in H^{k+1}(\Omega)$ and $u_h \in V_h^k$ for $k \in \{3,4\}$.
		There exists a positive constant $C$ such that:
		\begin{equation}
			\norm{u - u_h}_{h,*} \leq C h^{k-1} |u|_{k+1, \Omega}. \label{eq:energy_est}
		\end{equation}
	\end{theorem}
	\begin{proof}
		Let $u_I = \mathcal{I}_h^k u \in V_{h,\infty}^k$ be the quasi-interpolant
		from Definition~\ref{def:interp}.  Decompose the exact error as
		$\varepsilon = u - u_h = (u - u_I) - (u_h - u_I) =: \eta - \xi$,
		where $\eta$ is the interpolation error and $\xi \in V_h^k$ is the
		discrete error.
		By the coercivity (Theorem~\ref{thm:stability}) and the Galerkin
		orthogonality (Lemma~\ref{lem:orthogonality}), we have
		\begin{equation}
		\alpha \|\xi\|_h^2 \le A_h(\xi,\xi)
		= A_h(u_h-u_I,\xi)
		= A_h(u-u_I,\xi) - A_h(u-u_h,\xi)
		= A_h(\eta,\xi).
		\end{equation}
		The crucial observation is that the penalty part $d_h(\eta,\xi)$ vanishes identically.
		Indeed, $\eta = u - u_I$ has zero projected jumps because
		$u$ is smooth and $u_I \in V_{h,\infty}^k$.
		Consequently, the mixed continuity bound of Lemma~\ref{lem:mixed_cont}
		involves a constant $M$ that is \emph{completely independent} of the
		penalty parameters $\gamma_1,\gamma_2$.
		Thus,
		\[
		|A_h(\eta,\xi)| \le M \|\eta\|_{h,*} \|\xi\|_h,
		\]
		and dividing by $\|\xi\|_h$ yields
		$\|\xi\|_h \le \frac{M}{\alpha} \|\eta\|_{h,*}$.
		
		For the discrete function $\xi$, the boundary terms in the extended norm can be bounded by the interior terms via standard inverse inequalities, yielding $\norm{\xi}_{h,*} \leq C_{inv}^* \norm{\xi}_h$ for a generic constant $C_{inv}^* > 0$.
		Applying the triangle inequality gives:
		\begin{equation}
			\begin{aligned}
				\norm{u - u_h}_{h,*} &\leq \norm{\eta}_{h,*} + \norm{\xi}_{h,*} \\
				&\leq \norm{\eta}_{h,*} + C_{inv}^* \norm{\xi}_h \\
				&\leq \left(1 + C_{inv}^* \frac{M}{\alpha}\right) \norm{\eta}_{h,*}.
			\end{aligned}
		\end{equation}
		Finally, applying the interpolation bound \eqref{eq:interp_bound} from Lemma~\ref{lem:interpolation}, we conclude:
		\begin{equation}
				\norm{u - u_h}_{h,*} \leq \left(1 + C_{inv}^* \frac{M}{\alpha}\right) C h^{k-1} |u|_{k+1, \Omega}.
		\end{equation}
	\end{proof}

	\begin{theorem}[$L^2$ Norm Error Estimate] \label{thm:l2_error}
		Assume that the domain $\Omega$ admits the full $H^4$ elliptic regularity for the dual biharmonic problem (this holds on smooth domains but generally fails on polygons; see Remark~\ref{rem:dual_regularity}). Under this assumption the dual solution $\phi$ satisfies $\norm{\phi}_{4, \Omega} \leq C_{reg} \norm{\psi}_{0, \Omega}$.
		Assuming the penalty parameters $\gamma_1, \gamma_2$ are sufficiently large, the following optimal $L^2$ error estimate holds:
		\begin{equation}
			\norm{u - u_h}_{0, \Omega} \leq C h^{k+1} |u|_{k+1, \Omega}. \label{eq:l2_est}
		\end{equation}
	\end{theorem}
		
	\begin{proof}
		We employ the Aubin-Nitsche duality argument. Consider the dual problem $\Delta^2 \phi = u - u_h$.
		Let $\varepsilon = u - u_h$ denote the exact error. Testing the discrete formulation with the dual solution $\phi$, and letting $\phi_I = \mathcal{I}_h^k \phi$ be its interpolant (see Definition~\ref{def:interp}), the Galerkin orthogonality gives:
		
		\begin{equation}
			\norm{\varepsilon}_{0, \Omega}^2 = A_h(\varepsilon, \phi) = A_h(\varepsilon, \eta_\phi), \label{eq:aubin_nitsche}
		\end{equation}
		where $\eta_\phi = \phi - \phi_I$.
	
		By the elliptic regularity assumption, we have $|\phi|_{4, \Omega} \leq C_{reg} \norm{\varepsilon}_{0, \Omega}$.
		First, for the volume term, applying the Cauchy-Schwarz inequality and the global interpolation estimate yields:
		\begin{equation*}
			|a_h(\varepsilon, \eta_\phi)| \leq \norm{\varepsilon}_h \norm{\eta_\phi}_h \leq (C h^{k-1} |u|_{k+1,\Omega}) (C h^2 |\phi|_{4,\Omega}) \leq C h^{k+1} |u|_{k+1,\Omega} \norm{\varepsilon}_{0,\Omega}.
		\end{equation*}
		For the cross terms involving $T_e$, we apply the Cauchy-Schwarz inequality over the entire mesh skeleton $\mathcal{E}_h$:
		\begin{equation}
			\left| \sum_{e \in \mathcal{E}_h} \int_e \avg{T_e(\varepsilon)} \jmp{\eta_\phi} ds \right| \leq \left( \sum_{e \in \mathcal{E}_h} h_e^3 \norm{\avg{T_e(\varepsilon)}}_{0,e}^2 \right)^{1/2} \left( \sum_{e \in \mathcal{E}_h} h_e^{-3} \norm{\jmp{\eta_\phi}}_{0,e}^2 \right)^{1/2}.
		\end{equation}
		By the definition of the extended norm and the energy error estimate in Theorem~\ref{thm:energy_error}, the first global term is bounded by
		\begin{equation}
			\norm{\varepsilon}_{h,*} \leq C h^{k-1} |u|_{k+1,\Omega}.
		\end{equation}
	
		For the jump term of $\eta_\phi$, we proceed locally on each edge $e \in \mathcal{E}_h$. Let $E_e = K_1 \cup K_2$ denote the associated edge patch (the union of the one or two elements containing $e$).
		Since $\eta_\phi$ is continuous at the vertices, let $P$ be the linear Lagrange finite element interpolant of  $\eta_\phi$, then $P \in H_0^1(\Omega)$. Since $P$ is continuous across $e$, we have $\jmp{\eta_\phi}=\jmp{\eta_\phi-P}.$
		By the triangle inequality, $\|\jmp{\eta_\phi}\|_{0,e} \le \|(\eta_\phi - P)|_{K_1}\|_{0,e} + \|(\eta_\phi - P)|_{K_2}\|_{0,e}$. Applying the element-wise trace inequality \eqref{eq:trace_ineq} on $K_i$ ($i=1,2$) gives
		\[
		\|(\eta_\phi - P)|_{K_i}\|_{0,e} \le C \left( h_{K_i}^{-1/2} \|\eta_\phi - P\|_{0,K_i} + h_{K_i}^{1/2} |\eta_\phi - P|_{1,K_i} \right).
		\]
		Using the standard interpolation error estimate for the linear Lagrange interpolant on $K_i$, $\|\eta_\phi - P\|_{0,K_i} + h_{K_i} |\eta_\phi - P|_{1,K_i} \le C h_{K_i}^2 |\eta_\phi|_{2,K_i}$, we obtain
		\[
		\|(\eta_\phi - P)|_{K_i}\|_{0,e} \le C h_{K_i}^{3/2} |\eta_\phi|_{2,K_i}.
		\]
		Summing this over all edges and noting that each element acts as a patch for at most three edges, we obtain
		\[
		\sum_{e\in\mathcal E_h} h_e^{-3} \|\jmp{\eta_\phi}\|_{0,e}^2 \le C \sum_{K \in \mathcal{T}_h} |\eta_\phi|_{2,K}^2 = C |\eta_\phi|_{2,\mathcal{T}_h}^2.
		\]
		Invoking the global interpolation estimate of Lemma 3.3, $|\eta_\phi|_{2,\mathcal{T}_h} \le C h^2 |\phi|_{4,\Omega}$, and the elliptic regularity estimate,  we obtain
		\begin{equation}\label{eq:jump_bound}
			\Bigl( \sum_{e\in\mathcal{E}_h} h_e^{-3} \norm{\jmp{\eta_\phi}}_{0,e}^2 \Bigr)^{1/2} \le C h^2 |\phi|_{4,\Omega} \le C h^2 \|\varepsilon\|_{0,\Omega}.
		\end{equation}
		Combining \eqref{eq:jump_bound} with the energy norm error estimate $\|\varepsilon\|_{h,*}\le C h^{k-1}|u|_{k+1,\Omega}$, the cross term can be optimally bounded by:
		\begin{equation}
			\left| \sum_{e \in \mathcal{E}_h} \int_e \avg{T_e(\varepsilon)} \jmp{\eta_\phi} ds \right| \le C h^{k+1} |u|_{k+1,\Omega} \|\varepsilon\|_{0,\Omega}.
		\end{equation}
		
		For the other cross term, applying the Cauchy-Schwarz inequality yields:
		\begin{equation}
			\left| \sum_{e \in \mathcal{E}_h} \int_e \avg{T_e(\eta_\phi)} \jmp{\varepsilon} ds \right| \leq \left( \sum_{e \in \mathcal{E}_h} h_e^3 \norm{\avg{T_e(\eta_\phi)}}_{0,e}^2 \right)^{1/2} \left( \sum_{e \in \mathcal{E}_h} h_e^{-3} \norm{\jmp{\varepsilon}}_{0,e}^2 \right)^{1/2}.
		\end{equation}
		For the dual interpolation error $\eta_\phi = \phi - \mathcal{I}_h^k \phi$, Lemma~\ref{lem:interpolation} gives
		$\|\eta_\phi\|_{h,*} \le C h^{k-1} |\phi|_{k+1,\Omega}$. Although $\phi \in H^4(\Omega)$ by elliptic regularity, we can obtain the saturated bound
		\[
		\|\eta_\phi\|_{h,*} \le C h^2 |\phi|_{4,\Omega}
		\le C h^2 \|\varepsilon\|_{0,\Omega}.
		\]
		The discrete error component is similar to the above:
		\[
		\left( \sum_{e \in \mathcal{E}_h} h_e^{-3} \norm{\jmp{\varepsilon}}_{0,e}^2 \right)^{1/2}\leq \left( C \sum_{K \in \mathcal{T}_h} |u-u_h|_{2,K}^2  \right)^{1/2} \leq C h^{k-1}|u|_{k+1,\Omega}.
		\]
		This leads to
		\begin{equation}
			\begin{aligned}
				\left| \sum_{e \in \mathcal{E}_h} \int_e \avg{T_e(\eta_\phi)} \jmp{\varepsilon} ds \right| &\leq C (h^2 \norm{\varepsilon}_{0,\Omega}) (h^{k-1} |u|_{k+1,\Omega}) \\ 
				& = C h^{k+1} |u|_{k+1,\Omega} \norm{\varepsilon}_{0,\Omega}.
			\end{aligned}
		\end{equation}
		
		The terms involving $\jmp{\partial_n\eta_\phi}$ and $\jmp{\partial_n\varepsilon}$ are estimated analogously. Furthermore, since $\eta_\phi$ inherently possesses zero projected jumps by the definition of the interpolation operator, the penalty contribution vanishes identically, i.e., $d_h(\varepsilon, \eta_\phi) \equiv 0$. Collecting all contributions yields
		\[
		\|\varepsilon\|_{0,\Omega}^2
		\le
		C h^{k+1}
		|u|_{k+1,\Omega}
		\|\varepsilon\|_{0,\Omega}.
		\]
		Dividing both sides by $\norm{\varepsilon}_{0, \Omega}$ completes the proof.
	\end{proof}

	The estimate \eqref{eq:l2_est} shows that the discrete solution $u_h$ achieves the optimal $L^2$ convergence order $\mathcal{O}(h^{k+1})$.  
	In the asymptotic regime of large penalty parameters, Section \ref{subsec:penalty_evolution} further reveals that the $L^2$ error is dominated by the best approximation error of  $V_{h,\infty}^k$ (i.e., the optimal nonconforming space $B_h^k$), while the penalty contribution becomes negligible. Hence the method achieves optimal accuracy uniformly with respect to the penalty magnitude.

	\begin{remark}[Energy minimization and quasi-optimality]
		\label{rmk:energy_min}
		Since $A_h(\cdot,\cdot)$ is symmetric and coercive on $V_h^k$, the discrete solution $u_h$ is the unique minimizer over $V_h^k$ of the discrete energy functional $J(v_h) = \frac{1}{2}A_h(v_h,v_h) - (f,v_h)$. Equivalently, $u_h$ satisfies the Galerkin orthogonality $A_h(u-u_h, v_h)=0$ for all $v_h\in V_h^k$ (as proven in Lemma~\ref{lem:orthogonality}).
		By standard finite element arguments, for any arbitrary $v_h \in V_h^k$, we have:
		\begin{equation*}
			\alpha \|u_h - v_h\|_h^2 \le A_h(u_h - v_h, u_h - v_h) = A_h(u - v_h, u_h - v_h) \le M \|u - v_h\|_{h,*} \|u_h - v_h\|_h.
		\end{equation*}
		Dividing by $\|u_h - v_h\|_h$ and applying the triangle inequality together with the discrete inverse trace estimate yields the quasi-optimal bound:
		\[
		\|u-u_h\|_{h,*} \le \Bigl(1 + C_{\rm inv}^*\frac{M}{\alpha}\Bigr)
		\inf_{v_h\in V_h^k} \|u-v_h\|_{h,*}.
		\]
		The assertion that this bound strictly translates to an optimal approximation error stems from the nested structure of our spaces. For any \emph{fixed} $\gamma_1, \gamma_2$ satisfying the basic coercivity thresholds, the total error $\|u - u_h\|_{h,*}$ uniformly converges at the optimal $\mathcal{O}(h^{k-1})$ rate as $h \to 0$.  Since the limit space $V_{h,\infty}^k$ is a subspace of $V_h^k$, and $V_{h,\infty}^k \equiv B_h^k$, the infimum over the broader space $V_h^k$ is at least as small as the interpolation error within $B_h^k$. Because $B_h^k$ possesses the optimal approximation property $\mathcal{O}(h^{k-1})$, the quasi-optimal estimate guarantees the optimal convergence rate.
	\end{remark}

	\begin{remark}[Polygonal domains and dual regularity]\label{rem:dual_regularity}
		The full $H^4$ elliptic regularity required for the dual problem in Theorem~\ref{thm:l2_error} holds when $\Omega$ is a smooth domain (e.g., of class $C^{3,1}$).  On polygonal domains, however, corner singularities generally limit the dual solution $\phi$ to a strictly weaker Sobolev space.  For the clamped biharmonic problem on a convex square, the leading singularity in polar coordinates behaves as $r^\lambda$ with $\operatorname{Re}(\lambda) \approx 2.739$ 
		(see~\cite{Grisvard1985} and~\cite{Blum1980}), which restricts the dual solution $\phi$ to $H^{3.739-\varepsilon}(\Omega)$.
		Consequently, a rigorous application of the Aubin--Nitsche duality argument yields an $L^2$ convergence rate bounded by $\mathcal O(h^{k-1 + 1.739}) = \mathcal O(h^{k+0.739})$, rather than the optimal $\mathcal O(h^{k+1})$ predicted under full $H^4$ regularity.
		Despite this theoretical reduction, the numerical experiments on the square domain (Section~\ref{subsec:square_convergence}) still exhibit clean $\mathcal O(h^{k+1})$ convergence.  This is a well-known superconvergence effect on structured grids and does not contradict the locking-free property of the proposed scheme.
	\end{remark}

	\subsection{Evolution of the Penalty Parameter and Locking-Free Mechanism}
	\label{subsec:penalty_evolution}
	
    The behavior of the discrete solution as the penalty parameters vary can be summarized as follows.
		
	If the penalty parameters are smaller than the thresholds required in Theorem~\ref{thm:stability}, the bilinear form loses coercivity and the discrete system may become unstable.  
	Numerical experiments (Figures~\ref{fig:p4_s_combined}--\ref{fig:p3_l_combined}) confirm the oscillations for $\lambda < 10^2$, which correspond to this non-coercive range.
			
	When the penalty parameters just exceed the thresholds, the coercivity constant $\alpha$ in~\eqref{eq:alpha_expression} is small ($\alpha \to 0^+$). The quasi-optimal constant $C_{\rm eff} := 1 + C_{\rm inv}^* \frac{M}{\alpha}$ consequently becomes large, and the \emph{magnitude} of the error upper bound in~\eqref{eq:energy_est} degenerates. Nevertheless, the asymptotic convergence \emph{rate} $\mathcal O(h^{k-1})$ with respect to mesh refinement is entirely unaffected, because the degenerate constant $C_{\rm eff}$ is strictly independent of the mesh size $h$.
	
	The stabilization parameters $\gamma_1$ and $\gamma_2$ are \emph{dimensionless} constants. Because the proper physical mesh scalings $h_e^{-3}$ and $h_e^{-1}$ are intrinsically incorporated into the penalty operators in \eqref{eq:discrete_operators}, one does \emph{not} need to scale the penalty parameters with the mesh size (e.g., forcing $\gamma \sim h^{-3}$) to achieve convergence. As established in Remark \ref{rem:penalty_error}, the $\mathcal{O}(\gamma^{-1/2})$ bound merely quantifies how closely the IPDG solution shadows the limit nonconforming solution $u_{h,\infty}$ on a \emph{fixed} mesh as $\gamma_1, \gamma_2 \to \infty$.

	As the penalty parameters grow, the coercivity constant approaches $c_{\rm eq}/2$, and the continuity constant $M$ is independent of the penalty.
	Thus, for large penalty parameters, the penalty contribution to the error becomes negligible, ensuring the method is dominated solely by the optimal approximation power of $B_h^k$ and remains completely locking-free.

	\section{Numerical Experiments}
	\label{sec:numerical}
	In this section, we present numerical experiments to validate the theoretical analysis. The tests cover optimal convergence on regular meshes, robustness on non-convex domains, and sensitivity to the penalty parameters.
	
	To achieve this, we conduct our numerical simulations on two representative computational domains: a standard convex square domain and a non-convex L-shaped domain. The initial coarse triangulations for both domains are illustrated in Figure \ref{fig:initial_mesh_combined}. In the subsequent convergence tests, these initial meshes undergo successive uniform refinements to facilitate the asymptotic error analysis.
	
	\begin{figure}[htbp]
		\centering
		\subfloat[The initial mesh of the uniform triangulation]{\includegraphics[width=0.38\textwidth]{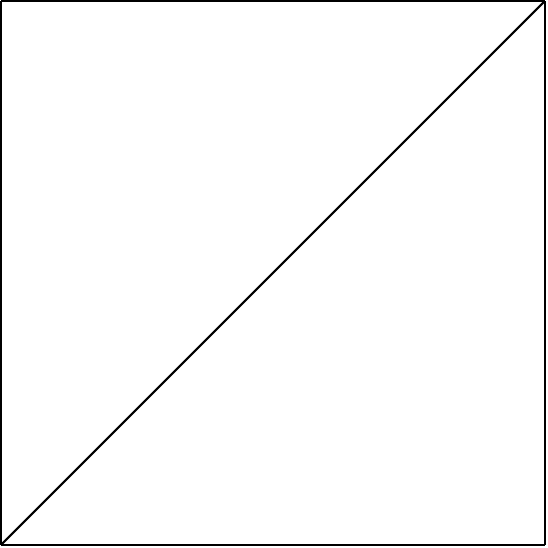}}
		\hspace{0.1\textwidth}
		\subfloat[The initial mesh of the triangulation for the L-shape domain]{\includegraphics[width=0.38\textwidth]{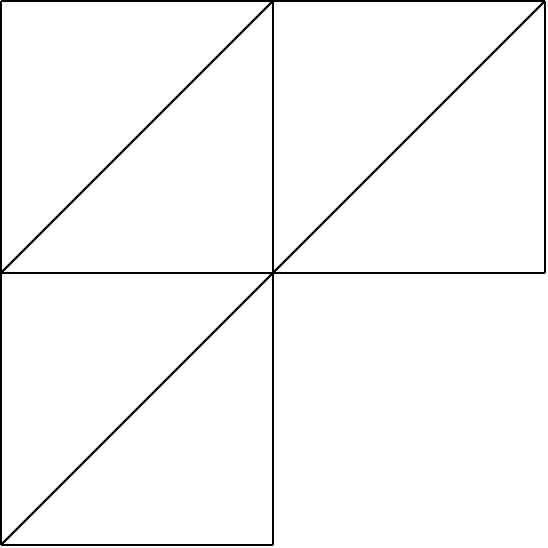}}
		\caption{The initial mesh.}
		\label{fig:initial_mesh_combined}
	\end{figure}

	\subsection{Experimental Setup and Error Measures}
	\label{subsec:setup}
	
	All numerical experiments are implemented and executed in MATLAB.
	The large sparse linear algebraic systems resulting from the discretization are solved using a direct solver.
	To assess the accuracy of our numerical approximations, we measure the errors in the discrete $L^2$ norm, the broken $H^1$ semi-norm, and the broken $H^2$ semi-norm.
	Let $u$ be the exact solution and $u_h \in V_h^k$ be the numerical solution obtained by the IPDG method.
	The error norms are defined cell-wise as follows:
	\begin{align}
		e_{L^2} &= \norm{u - u_h}_{0,\Omega}, \label{eq:error_l2} \\
		e_{H^1} &= \left( \sum_{T \in \mathcal{T}_h} \norm{\nabla(u - u_h)}_{0,T}^2 \right)^{1/2}, \label{eq:error_h1} \\
		e_{H^2} &= \left( \sum_{T \in \mathcal{T}_h} \norm{\nabla^2(u - u_h)}_{0,T}^2 \right)^{1/2}. \label{eq:error_h2}
	\end{align}
	The experimental convergence rate (Order) is calculated using the standard logarithmic ratio formula for two adjacent meshes:
	\begin{equation}
		\text{Order} = \frac{\log(e_{h_1} / e_{h_2})}{\log(h_1 / h_2)}, \label{eq:error_order}
	\end{equation}
	where $e_{h_1}$ and $e_{h_2}$ denote the numerical errors corresponding to mesh sizes $h_1$ and $h_2$, respectively.
	
	In the theoretical analysis, $\gamma_1$ and $\gamma_2$ control the function and normal derivative jumps, respectively.
	In the subsequent numerical implementation, we unify these stabilization parameters for simplicity by setting $\gamma_1 = \gamma_2 = \lambda$.
	Unless otherwise specified, we fix $\lambda = 1000$ to strictly satisfy the theoretical coercivity threshold derived in Theorem~\ref{thm:stability} (i.e., $\lambda > \max\{C_1^2, C_2^2\}$).
	
	In the convergence tables, $h$ denotes the mesh size of the coarsest triangulation used for each domain; subsequent meshes are
	obtained by uniform refinement (each triangle is subdivided into four congruent subtriangles).  For the unit square,
	$h_{\text{initial}}=1$; for the L-shaped domain, $h_{\text{initial}}=1$ corresponds to the triangulation shown in Figure~\ref{fig:initial_mesh_combined}(b), whose maximal element diameter is approximately $0.5\sqrt{2}$.
	
	\subsection{Accuracy and Optimal Convergence on the Square Domain}
	\label{subsec:square_convergence}
	In this subsection, we numerically verify the \textit{a priori} error estimates. We consider the biharmonic equation on a unit square domain $\Omega = (0,1)^2$ equipped with uniform structured meshes. To verify the approximation capabilities, we test both polynomial and oscillatory exact solutions.
	
	\subsubsection{Test 1: Smooth Polynomial Solution}
	\label{subsubsec:poly_test}
	
	We first construct a smooth polynomial exact solution designed to satisfy the homogeneous clamped boundary conditions:
	\begin{equation}
		u(x,y) = x^2(1-x)^2 y^2(1-y)^2.
		\label{eq:exact_poly}
	\end{equation}
	The corresponding right-hand side forcing term $f(x,y)$ is derived by applying the biharmonic operator $\Delta^2$ to the exact solution.
	
	Table~\ref{tab:poly_square_p3_convergence} and Table~\ref{tab:poly_square_p4_convergence} summarize the numerical errors and convergence rates for the $P_3$ and $P_4$ elements, respectively. The ``--'' markers in the first row of all convergence tables indicate the absence of a coarser mesh.
	As the mesh is uniformly refined, we observe that the $P_3$ element achieves the optimal 2nd-order asymptotic convergence in the energy norm ($H^2$-norm) and 4th-order in the $L^2$ norm.
	Similarly, the $P_4$ element exhibits the expected optimal convergence rates of 3rd-order in the energy norm and 5th-order in the $L^2$ norm.
	
	These results agree well with the theoretical predictions, verifying that the reduced-order penalty strategy achieves optimal approximation without loss of accuracy.
	
	\begin{table}[htbp]
		\footnotesize\tabcolsep 6pt
		\begin{center}
			\caption{Numerical errors and convergence rates of the $P_3$ IPDG method for the polynomial solution on the square domain ($\lambda = 1000$).}
			\label{tab:poly_square_p3_convergence}
			\begin{tabular}{@{} c c c c c c c c @{}}
				\toprule
				\multirow{2}{*}{Level} & \multirow{2}{*}{$h$ (refinement level)} & \multicolumn{2}{c}{$L^2$-Norm} & \multicolumn{2}{c}{$H^1$-Norm} & \multicolumn{2}{c}{$H^2$-Norm} \\
				\cmidrule(lr){3-4} \cmidrule(lr){5-6} \cmidrule(l){7-8}
				& & Error & Order & Error & Order & Error & Order \\
				\midrule
				1 & 1.000000 & $9.3638 \times 10^{-4}$ &   --   & $5.9091 \times 10^{-3}$ &   --   & $5.2839 \times 10^{-2}$ &   --   \\
				2 & 0.500000 & $5.5246 \times 10^{-4}$ & 0.76 & $3.5245 \times 10^{-3}$ & 0.75 & $3.4262 \times 10^{-2}$ & 0.63 \\
				3 & 0.250000 & $3.3807 \times 10^{-5}$ & 4.03 & $3.6928 \times 10^{-4}$ & 3.25 & $9.4290 \times 10^{-3}$ & 1.86 \\
				4 & 0.125000 & $1.5379 \times 10^{-6}$ & 4.46 & $4.0248 \times 10^{-5}$ & 3.20 & $2.3310 \times 10^{-3}$ & 2.02 \\
				5 & 0.062500 & $6.7533 \times 10^{-8}$ & 4.51 & $4.7295 \times 10^{-6}$ & 3.09 & $5.6760 \times 10^{-4}$ & 2.04 \\
				6 & 0.031250 & $3.7675 \times 10^{-9}$ & 4.16 & $5.8289 \times 10^{-7}$ & 3.02 & $1.3938 \times 10^{-4}$ & 2.03 \\
				\bottomrule
			\end{tabular}
		\end{center}
	\end{table}
	
	\begin{table}[htbp]
		\footnotesize\tabcolsep 6pt
		\begin{center}
			\caption{Numerical errors and convergence rates of the $P_4$ IPDG method for the polynomial solution on the square domain ($\lambda = 1000$).}
			\label{tab:poly_square_p4_convergence}
			\begin{tabular}{@{} c c c c c c c c @{}}
				\toprule
				\multirow{2}{*}{Level} 
				& \multirow{2}{*}{$h$ (refinement level)} & \multicolumn{2}{c}{$L^2$-Norm} & \multicolumn{2}{c}{$H^1$-Norm} & \multicolumn{2}{c}{$H^2$-Norm} \\
				\cmidrule(lr){3-4} \cmidrule(lr){5-6} \cmidrule(l){7-8}
				& & Error & Order & Error & Order & Error & Order \\
				\midrule
				1 & 1.000000 & $1.4422 \times 10^{-3}$ &   --    & $6.8338 \times 10^{-3}$ &   --   & $5.2896 \times 10^{-2}$ &   --   \\
				2 & 0.500000 & $2.4573 \times 10^{-4}$ &  2.55 & $3.4986 \times 10^{-3}$ &  0.97 & $5.8214 \times 10^{-2}$ & 0.14 \\
				3 & 0.250000 & $7.1550 \times 10^{-6}$ &  5.10 & $1.8568 \times 10^{-4}$ &  4.24 & $6.1515 \times 10^{-3}$ &  3.24  \\
				4 & 0.125000 & $2.1998 \times 10^{-7}$ &  5.02 & $1.1718 \times 10^{-5}$ &  3.99 & $7.9275 \times 10^{-4}$ &  2.96  \\
				5 & 0.062500 & $5.0360 \times 10^{-9}$ &  5.45 & $5.4694 \times 10^{-7}$ &  4.42 & $7.5421 \times 10^{-5}$ &  3.39  \\
				6 & 0.031250 & $1.0231 \times 10^{-10}$ &  5.62 & $2.3556 \times 10^{-8}$ &  4.54 & $6.7422 \times 10^{-6}$ &  3.48  \\
				\bottomrule
			\end{tabular}
		\end{center}
	\end{table}
	
	\subsubsection{Test 2: Smooth Trigonometric Solution}
	\label{subsubsec:trig_test}
	
	To exclude potential superconvergence phenomena specific to simple polynomial solutions, we further test a highly oscillatory trigonometric exact solution:
	\begin{equation}
		u(x,y) = \sin(2\pi x)\sin(2\pi y).
		\label{eq:exact_trig}
	\end{equation}
	The corresponding right-hand side function is derived by substituting the exact solution into the governing equation, yielding $f(x,y) = 64\pi^4\sin(2\pi x)\sin(2\pi y)$.
	
	The computational results are reported in Table~\ref{tab:trig_square_p3_convergence} and Table~\ref{tab:trig_square_p4_convergence}. Consistent with the polynomial test case, the $P_3$ and $P_4$ reduced-penalty schemes consistently maintain their respective optimal convergence rates even when approximating functions with rich high-frequency components.
	Specifically, at the finest mesh levels, the $P_4$ element achieves optimal 4th-order convergence in the $H^1$ norm and asymptotically approaches 3rd-order in the $H^2$ norm.
	This confirms that the asymptotic convergence rates predicted by the theory are attained for generic smooth solutions.
	
	\begin{table}[htbp]
		\footnotesize\tabcolsep 6pt
		\begin{center}
			\caption{Numerical errors and convergence rates of the $P_3$ IPDG method for the trigonometric solution on the square domain ($\lambda = 1000$).}
			\label{tab:trig_square_p3_convergence}
			\begin{tabular}{@{} c c c c c c c c @{}}
				\toprule
				\multirow{2}{*}{Level} & \multirow{2}{*}{$h$ (refinement level)} & \multicolumn{2}{c}{$L^2$-Norm} & \multicolumn{2}{c}{$H^1$-Norm} & \multicolumn{2}{c}{$H^2$-Norm} \\
				\cmidrule(lr){3-4} \cmidrule(lr){5-6} \cmidrule(l){7-8}
				& & Error & Order & Error & Order & Error & Order \\
				\midrule
				1 & 1.000000 & $3.1955 \times 10^{-1}$ &   --   & $1.6136 \times 10^{0}$ &   --   & $1.4399 \times 10^{1}$ &   --   \\
				2 & 0.500000 & $1.2487 \times 10^{-1}$ & 1.36 & $8.0309 \times 10^{-1}$ & 1.01 & $7.9565 \times 10^{0}$ & 0.86 \\
				3 & 0.250000 & $5.6381 \times 10^{-3}$ & 4.47 & $8.1103 \times 10^{-2}$ & 3.31 & $2.1351 \times 10^{0}$ & 1.90 \\
				4 & 0.125000 & $3.1017 \times 10^{-4}$ & 4.18 & $9.2556 \times 10^{-3}$ & 3.13 & $5.3521 \times 10^{-1}$ & 2.00 \\
				5 & 0.062500 & $1.8499 \times 10^{-5}$ & 4.07 & $1.1711 \times 10^{-3}$ & 2.98 & $1.3480 \times 10^{-1}$ & 1.99 \\
				6 & 0.031250 & $1.2182 \times 10^{-6}$ & 3.92 & $1.5404 \times 10^{-4}$ & 2.93 & $3.4255 \times 10^{-2}$ & 1.98 \\
				\bottomrule
			\end{tabular}
		\end{center}
	\end{table}
	
	\begin{table}[htbp]
		\footnotesize\tabcolsep 6pt
		\begin{center}
			\caption{Numerical errors and convergence rates of the $P_4$ IPDG method for the trigonometric solution on the square domain ($\lambda = 1000$).}
			\label{tab:trig_square_p4_convergence}
			\begin{tabular}{@{} c c c c c c c c @{}}
				\toprule
				\multirow{2}{*}{Level} 
				& \multirow{2}{*}{$h$ (refinement level)} & \multicolumn{2}{c}{$L^2$-Norm} & \multicolumn{2}{c}{$H^1$-Norm} & \multicolumn{2}{c}{$H^2$-Norm} \\
				\cmidrule(lr){3-4} \cmidrule(lr){5-6} \cmidrule(l){7-8}
				& & Error & Order & Error & Order & Error & Order \\
				\midrule
				1 & 1.000000 & $3.4386 \times 10^{-1}$ &   --    & $1.6240 \times 10^{0}$ &   --   & $1.3475 \times 10^{1}$ &   --   \\
				2 & 0.500000 & $8.5245 \times 10^{-2}$ &  2.01 & $1.2386 \times 10^{0}$ &  0.39 & $2.1068 \times 10^{1}$ & 0.64 \\
				3 & 0.250000 & $1.2396 \times 10^{-3}$ &  6.10 & $3.2392 \times 10^{-2}$ &  5.26 & $1.0439 \times 10^{0}$ &  4.33  \\
				4 & 0.125000 & $2.1320 \times 10^{-5}$ &  5.86 & $1.2541 \times 10^{-3}$ &  4.69 & $8.8380 \times 10^{-2}$ &  3.56  \\
				5 & 0.062500 & $4.2288 \times 10^{-7}$ &  5.66 & $5.9729 \times 10^{-5}$ &  4.39 & $9.2184 \times 10^{-3}$ &  3.26  \\
				6 & 0.031250 & $1.1121 \times 10^{-8}$ &  5.25 & $3.3861 \times 10^{-6}$ &  4.14 & $1.0775 \times 10^{-3}$ &  3.10  \\
				\bottomrule
			\end{tabular}
		\end{center}
	\end{table}

	\subsection{Performance on Non-Convex L-Shaped Meshes}
	\label{subsec:lshape_robustness}
	
	Although the \textit{a priori} error estimates were derived under the assumption of full elliptic regularity on convex polygons, practical engineering applications often involve meshes generated over non-convex domains. This subsection aims to demonstrate the geometric robustness of our discontinuous Galerkin (DG) framework when dealing with non-convex boundaries. We solve the biharmonic equation on an L-shaped domain $\Omega = [-1,1]^2 \setminus ([0,1] \times [-1,0])$. 
	
	To verify the stability of the vertex-continuous topological constraints and the reduced-order penalization on non-convex meshes, we employ manufactured solutions prescribed via boundary data. In this subsection, we focus on a smooth polynomial solution and a highly challenging singular solution. Additionally, to further exclude potential superconvergence phenomena and confirm the method's behavior against high-frequency oscillations on non-convex domains, a supplementary test utilizing a smooth trigonometric solution is provided in \ref{sec:appendix_trig}. For the smooth polynomial and trigonometric test cases in this section, the penalty parameter is fixed at $\lambda = 1000$.
	
	\subsubsection{Test 1: Smooth Polynomial Solution}
	\label{subsubsec:poly_lshape}
	
	We first consider a polynomial exact solution designed for the L-shaped domain:
	\begin{equation}
		u(x,y) = (x^2-1)^2 x^2 (y^2-1)^2 y^2.
		\label{eq:exact_poly_lshape}
	\end{equation}
	
	The numerical errors and convergence rates for the $P_3$ and $P_4$ IPDG methods are summarized in Table~\ref{tab:poly_lshape_p3_convergence} and Table~\ref{tab:poly_lshape_p4_convergence}, respectively.
	Despite the presence of the non-convex corner, the $P_3$ element maintains robust convergence, achieving nearly 2nd-order accuracy in the energy norm ($H^2$-norm) and nearly 4th-order in the $L^2$ norm as the mesh is sufficiently refined.
	Similarly, the $P_4$ element reaches approximately 3rd-order convergence in the energy norm and 5th-order in the $L^2$ norm.
	Overall, the method successfully maintains robustness on the non-convex domain.
	
	\begin{table}[htbp]
		\footnotesize\tabcolsep 6pt
		\begin{center}
			\caption{Numerical errors and convergence rates of the $P_3$ IPDG method for the polynomial solution on the L-shaped domain ($\lambda = 1000$).}
			\label{tab:poly_lshape_p3_convergence}
			\begin{tabular}{@{} c c c c c c c c @{}}
				\toprule
				\multirow{2}{*}{Level} & \multirow{2}{*}{$h$ (refinement level)} & \multicolumn{2}{c}{$L^2$-Norm} & \multicolumn{2}{c}{$H^1$-Norm} & \multicolumn{2}{c}{$H^2$-Norm} \\
				\cmidrule(lr){3-4} \cmidrule(lr){5-6} \cmidrule(l){7-8}
				& & Error & Order & Error & Order & Error & Order \\
				\midrule
				1 & 1.000000 & $1.0502 \times 10^{-2}$ &   --   & $6.1573 \times 10^{-2}$ &   --   & $5.4032 \times 10^{-1}$ &   --   \\
				2 & 0.500000 & $4.4062 \times 10^{-3}$ & 1.25 & $3.4700 \times 10^{-2}$ & 0.83 & $3.7581 \times 10^{-1}$ & 0.52 \\
				3 & 0.250000 & $3.8380 \times 10^{-4}$ & 3.52 & $4.9340 \times 10^{-3}$ & 2.81 & $1.2435 \times 10^{-1}$ & 1.60 \\
				4 & 0.125000 & $1.9203 \times 10^{-5}$ & 4.32 & $5.7133 \times 10^{-4}$ & 3.11 & $3.2508 \times 10^{-2}$ & 1.94 \\
				5 & 0.062500 & $1.0887 \times 10^{-6}$ & 4.14 & $7.4266 \times 10^{-5}$ & 2.94 & $8.3235 \times 10^{-3}$ & 1.97 \\
				6 & 0.031250 & $7.7249 \times 10^{-8}$ & 3.82 & $9.8318 \times 10^{-6}$ & 2.92 & $2.1143 \times 10^{-3}$ & 1.98 \\
				\bottomrule
			\end{tabular}
		\end{center}
	\end{table}
	
	\begin{table}[htbp]
		\footnotesize\tabcolsep 6pt
		\begin{center}
			\caption{Numerical errors and convergence rates of the $P_4$ IPDG method for the polynomial solution on the L-shaped domain ($\lambda = 1000$).}
			\label{tab:poly_lshape_p4_convergence}
			\begin{tabular}{@{} c c c c c c c c @{}}
				\toprule
				\multirow{2}{*}{Level} & \multirow{2}{*}{$h$ (refinement level)} & \multicolumn{2}{c}{$L^2$-Norm} & \multicolumn{2}{c}{$H^1$-Norm} & \multicolumn{2}{c}{$H^2$-Norm} \\
				\cmidrule(lr){3-4} \cmidrule(lr){5-6} \cmidrule(l){7-8}
				& & Error & Order & Error & Order & Error & Order \\
				\midrule
				1 & 1.000000 & $1.0355 \times 10^{-1}$ &   --    & $8.0936 \times 10^{-1}$ &   --   & $6.7102 \times 10^{0}$ &   --    \\
				2 & 0.500000 & $2.0932 \times 10^{-3}$ &  5.63 & $2.0307 \times 10^{-2}$ &  5.32 & $2.9509 \times 10^{-1}$ &  4.51  \\
				3 & 0.250000 & $8.1562 \times 10^{-5}$ &  4.68 & $1.9912 \times 10^{-3}$ &  3.35 & $6.6686 \times 10^{-2}$ &  2.15  \\
				4 & 0.125000 & $2.9772 \times 10^{-6}$ &  4.78 & $1.5144 \times 10^{-4}$ &  3.72 & $1.0301 \times 10^{-2}$ &  2.69  \\
				5 & 0.062500 & $7.8602 \times 10^{-8}$ &  5.24 & $8.4385 \times 10^{-6}$ &  4.17 & $1.1634 \times 10^{-3}$ &  3.15  \\
				6 & 0.031250 & $1.7000 \times 10^{-9}$ &  5.53 & $3.8264 \times 10^{-7}$ &  4.46 & $1.0874 \times 10^{-4}$ &  3.42  \\
				\bottomrule
			\end{tabular}
		\end{center}
	\end{table}

	\subsubsection{Test 2: Singular Solution}
	\label{subsubsec:sing_lshape}
	
	To further demonstrate the universality of the proposed framework, we evaluate the method using a classic singular exact solution on the L-shaped domain, as constructed by Hu et al.~\cite{Hu2021}. This solution exhibits a strong geometric singularity at the reentrant corner (the origin).
	
	Let $\omega := 3\pi/2$, and $\alpha = 0.544483736782464$ is a non-characteristic root of $\sin^2(\alpha\omega)=\alpha^2\sin^2(\omega)$ with
	\[
	g_{\alpha,\omega}(\theta) = g_1\bigl(\cos((\alpha-1)\theta) - \cos((\alpha+1)\theta)\bigr)
	- g_2\left(\frac{1}{\alpha-1}\sin((\alpha-1)\theta) - \frac{1}{\alpha+1}\sin((\alpha+1)\theta)\right)
	\]
	and
	\[
	g_1 = \frac{1}{\alpha-1}\sin((\alpha-1)\omega) - \frac{1}{\alpha+1}\sin((\alpha+1)\omega),
	\qquad
	g_2 = \cos((\alpha-1)\omega) - \cos((\alpha+1)\omega).
	\]
	
	The load function $f = \Delta^2 u$ is derived by the exact solution
	\[
	u(x,y) = (1-x^2)^2(1-y^2)^2\bigl(\sqrt{x^2+y^2}\bigr)^{1+\alpha} g_{\alpha,\omega}(\theta).
	\]
	
	Due to the severe singularity at the reentrant corner, the global regularity of this exact solution is fundamentally limited, satisfying $u \in H^{2+\alpha-\epsilon}(\Omega)$ for any $\epsilon > 0$. According to standard finite element approximation theory, this low global regularity limits the convergence rates. For the primal IPDG formulation, regardless of the high polynomial degree employed ($P_3$ or $P_4$), the theoretical upper bounds for the convergence rates are restricted to $\mathcal{O}(h^\alpha) \approx \mathcal{O}(h^{0.544})$ in the $H^2$ energy norm and $\mathcal{O}(h^{1+\alpha}) \approx \mathcal{O}(h^{1.544})$ in the $H^1$ norm. Moreover, owing to the non-convexity of the domain, the dual regularity required for the Aubin-Nitsche duality argument is compromised, restricting the optimal $L^2$ norm convergence limit to roughly $\mathcal{O}(h^{2\alpha}) \approx \mathcal{O}(h^{1.088})$. To maintain discrete coercivity near the singularity, we increase the penalty parameter to $\lambda = 5000$ for this test.
	
	The corresponding numerical results for the $P_3$ and $P_4$ elements are reported in Table~\ref{tab:ipdg_lshape_p3_singular_convergence} and Table~\ref{tab:ipdg_lshape_p4_singular_convergence}, respectively. As the mesh is uniformly refined, the empirical convergence rates for both elements naturally deviate from their optimal polynomial orders and asymptotically approach the theoretical regularity limits. Specifically, at the finest mesh levels, the $H^2$ error convergence stabilizes at approximately $0.55$, which closely matches the $0.54$ asymptotic order of the stress tensor error reported by Hu et al.~\cite{Hu2021} using a mixed finite element method. Concurrently, the $L^2$ error convergence approaches $1.09 \sim 1.23$, consistent with the theoretical dual limit. The stable approximation of this severe singularity by the $P_4$ element, without inducing spurious oscillations, demonstrates the robustness of the proposed IPDG scheme under low regularity conditions.

	\begin{table}[htbp]
		\footnotesize\tabcolsep 6pt
		\begin{center}
			\caption{Numerical errors and convergence rates of the $P_3$ IPDG method for the singular solution on the L-shaped domain ($\lambda = 5000$).}
			\label{tab:ipdg_lshape_p3_singular_convergence}
			\begin{tabular}{@{} c c c c c c c c @{}}
				\toprule
				\multirow{2}{*}{Level} 
				& \multirow{2}{*}{$h$ (refinement level)} & \multicolumn{2}{c}{$L^2$-Norm} & \multicolumn{2}{c}{$H^1$-Norm} & \multicolumn{2}{c}{$H^2$-Norm} \\
				\cmidrule(lr){3-4} \cmidrule(lr){5-6} \cmidrule(l){7-8}
				& & Error & Order & Error & Order & Error & Order \\
				\midrule
				1 & 1.000000 & $4.7380 \times 10^{-1}$ &   --    & $1.8057 \times 10^{0}$ &   --   & $1.1247 \times 10^{1}$ &   --   \\
				2 & 0.500000 & $8.0178 \times 10^{-2}$ &  2.56 & $4.0967 \times 10^{-1}$ &  2.14 & $4.7878 \times 10^{0}$ &  1.23 \\
				3 & 0.250000 & $8.8674 \times 10^{-3}$ &  3.18 & $6.3204 \times 10^{-2}$ &  2.70 & $1.5985 \times 10^{0}$ &  1.58 \\
				4 & 0.125000 & $1.5336 \times 10^{-3}$ &  2.53 & $1.5276 \times 10^{-2}$ &  2.05 & $7.6559 \times 10^{-1}$ &  1.06 \\
				5 & 0.062500 & $4.9625 \times 10^{-4}$ &  1.63 & $5.1525 \times 10^{-3}$ &  1.57 & $4.9359 \times 10^{-1}$ &  0.63 \\
				6 & 0.031250 & $2.1152 \times 10^{-4}$ &  1.23 & $1.8619 \times 10^{-3}$ &  1.47 & $3.3528 \times 10^{-1}$ &  0.56 \\
				\bottomrule
			\end{tabular}
		\end{center}
	\end{table}
	
	\begin{table}[htbp]
		\footnotesize\tabcolsep 6pt
		\begin{center}
			\caption{Numerical errors and convergence rates of the $P_4$ IPDG method for the singular solution on the L-shaped domain ($\lambda = 5000$).}
			\label{tab:ipdg_lshape_p4_singular_convergence}
			\begin{tabular}{@{} c c c c c c c c @{}}
				\toprule
				\multirow{2}{*}{Level} 
				& \multirow{2}{*}{$h$ (refinement level)} & \multicolumn{2}{c}{$L^2$-Norm} & \multicolumn{2}{c}{$H^1$-Norm} & \multicolumn{2}{c}{$H^2$-Norm} \\
				\cmidrule(lr){3-4} \cmidrule(lr){5-6} \cmidrule(l){7-8}
				& & Error & Order & Error & Order & Error & Order \\
				\midrule
				1 & 1.000000 & $2.2592 \times 10^{-1}$ &   --    & $9.1032 \times 10^{-1}$ &   --   & $7.6781 \times 10^{0}$ &   --   \\
				2 & 0.500000 & $1.2705 \times 10^{-2}$ &  4.15 & $8.7199 \times 10^{-2}$ &  3.38 & $1.6454 \times 10^{0}$ &  2.22 \\
				3 & 0.250000 & $2.7436 \times 10^{-3}$ &  2.21 & $2.2779 \times 10^{-2}$ &  1.94 & $7.4969 \times 10^{-1}$ &  1.13 \\
				4 & 0.125000 & $1.1813 \times 10^{-3}$ &  1.22 & $8.4778 \times 10^{-3}$ &  1.43 & $5.0127 \times 10^{-1}$ &  0.58 \\
				5 & 0.062500 & $5.4954 \times 10^{-4}$ &  1.10 & $3.3472 \times 10^{-3}$ &  1.34 & $3.4299 \times 10^{-1}$ &  0.55 \\
				6 & 0.031250 & $2.5869 \times 10^{-4}$ &  1.09 & $1.3877 \times 10^{-3}$ &  1.27 & $2.3502 \times 10^{-1}$ &  0.55 \\
				\bottomrule
			\end{tabular}
		\end{center}
	\end{table}

	\subsection{Sensitivity Analysis of the Penalty Parameters}
	\label{subsec:penalty_sensitivity}
	A known drawback of standard full-penalty high-order DG methods is that a large penalty parameter imposes overly strict continuity constraints on element interfaces, typically causing numerical locking. As highlighted during the method's construction, our adaptive reduced-order projection strategy (utilizing $\mathcal{P}^{k-3}$ for function value jumps and $\mathcal{P}^{k-2}$ for normal derivative jumps) theoretically eliminates this over-constraint. This subsection provides direct empirical evidence supporting this claim.
	
	To systematically investigate the parameter robustness of the proposed IPDG method, we fix the mesh size at $h = 0.0625$ and scan the penalty parameter $\lambda$ over a wide range:
	
	\begin{equation*}
		\lambda \in \{1, 10, 50, 100, 200, 500, 1000, 5000, 10000, 20000, 30000, 50000, 80000, 100000\}.
	\end{equation*}
	In the subsequent convergence order plots (the bottom rows of Figures~\ref{fig:p4_s_combined} and~\ref{fig:p4_l_combined}, as well as Figures~\ref{fig:p3_s_combined} and~\ref{fig:p3_l_combined} in \ref{sec:appendix_sensitivity}), the legends (e.g., $h_1/h_2$, $h_2/h_3$) specifically denote the empirical convergence rates calculated between two successive uniformly refined meshes across the entire spectrum of penalty parameters.
	Due to space constraints, the detailed sensitivity analysis for the $P_3$ element is provided in \ref{sec:appendix_sensitivity}.
	Here, we focus on the sensitivity analysis for the $P_4$ element.
	
	The theoretical optimal convergence rates for the $P_4$ element are $\mathcal{O}(h^5)$ for the $L^2$ norm, $\mathcal{O}(h^4)$ for the $H^1$ norm, and $\mathcal{O}(h^3)$ for the $H^2$ energy norm. The performance on the regular square domain is depicted in Figure~\ref{fig:p4_s_combined}, while the performance on the L-shaped domain is illustrated in Figure~\ref{fig:p4_l_combined}.
	
	In the regime of small penalty parameters (e.g., $\lambda < 10^2$), minor fluctuations in both error magnitudes and convergence rates are observed. This is not a numerical artifact, but a direct consequence of violating the theoretical coercivity threshold. As established in Theorem~\ref{thm:stability}, the discrete bilinear form $A_h(v,v)$ maintains strict positive-definiteness only when the stabilization parameters exceed a critical lower bound ($\gamma > C^2$). When $\lambda$ is insufficient, the interfacial constraints are too weak to stabilize the system, leading to non-physical spurious modes.

	Conversely, we observe a highly robust, parameter-independent stable regime in the medium-to-large penalty range ($\lambda \ge 10^3$). In standard full-penalty IPDG formulations, increasing the penalty parameter to such extremes (e.g., $10^5$) typically causes severe numerical locking. Our method avoids this difficulty entirely.
	As $\lambda$ increases, the penalty terms enforce the projected jumps to be vanishingly small, effectively constraining the discrete solution to an $\mathcal{O}(\lambda^{-1/2})$-neighborhood of the constrained subspace $V_{h,\infty}^k$(see also the quantitative bound in Subsection~\ref{subsec:penalty_evolution}). 
	Since $V_{h,\infty}^k$ coincides with the optimal nonconforming space $B_h^k$, the approximation error remains dominated by the best approximation error in $B_h^k$, which is $\mathcal{O}(h^{k-1})$ uniformly in $\lambda$.
	The horizontal plateaus observed in the figures are the empirical manifestations of this theoretical guarantee: the method is completely locking-free.
	Furthermore, a comparative analysis between the convex square domain and the non-convex L-shaped domain highlights the geometric robustness of the spatial discretization. Despite the reentrant corner, the parameter-independent stable regime remains fully intact.
	
	\begin{figure}[p]
		\centering
		\subfloat[Polynomial exact solution]{\includegraphics[width=1.0\textwidth]{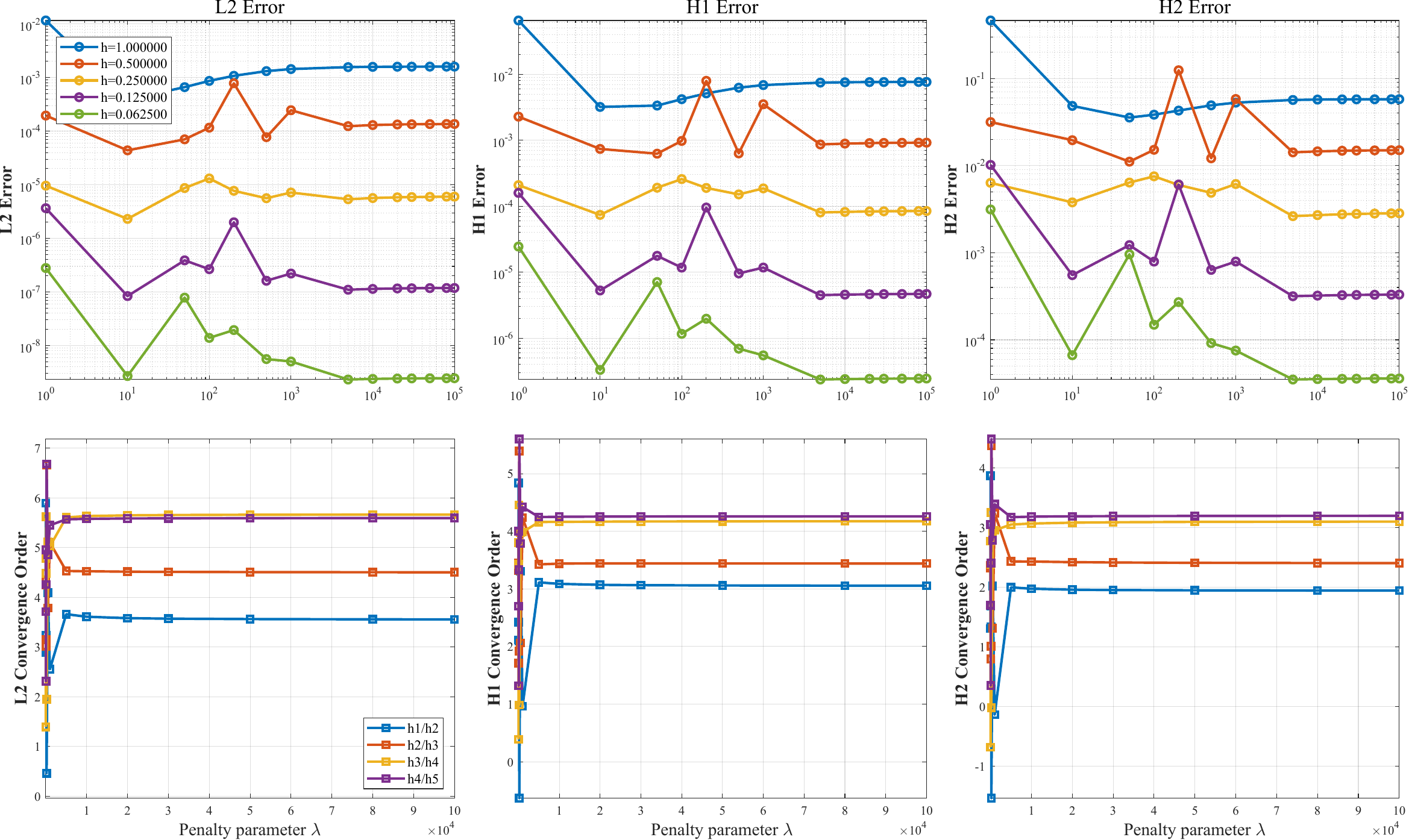}} \\
		\vspace{0.4cm}
		\subfloat[Trigonometric exact solution]{\includegraphics[width=1.0\textwidth]{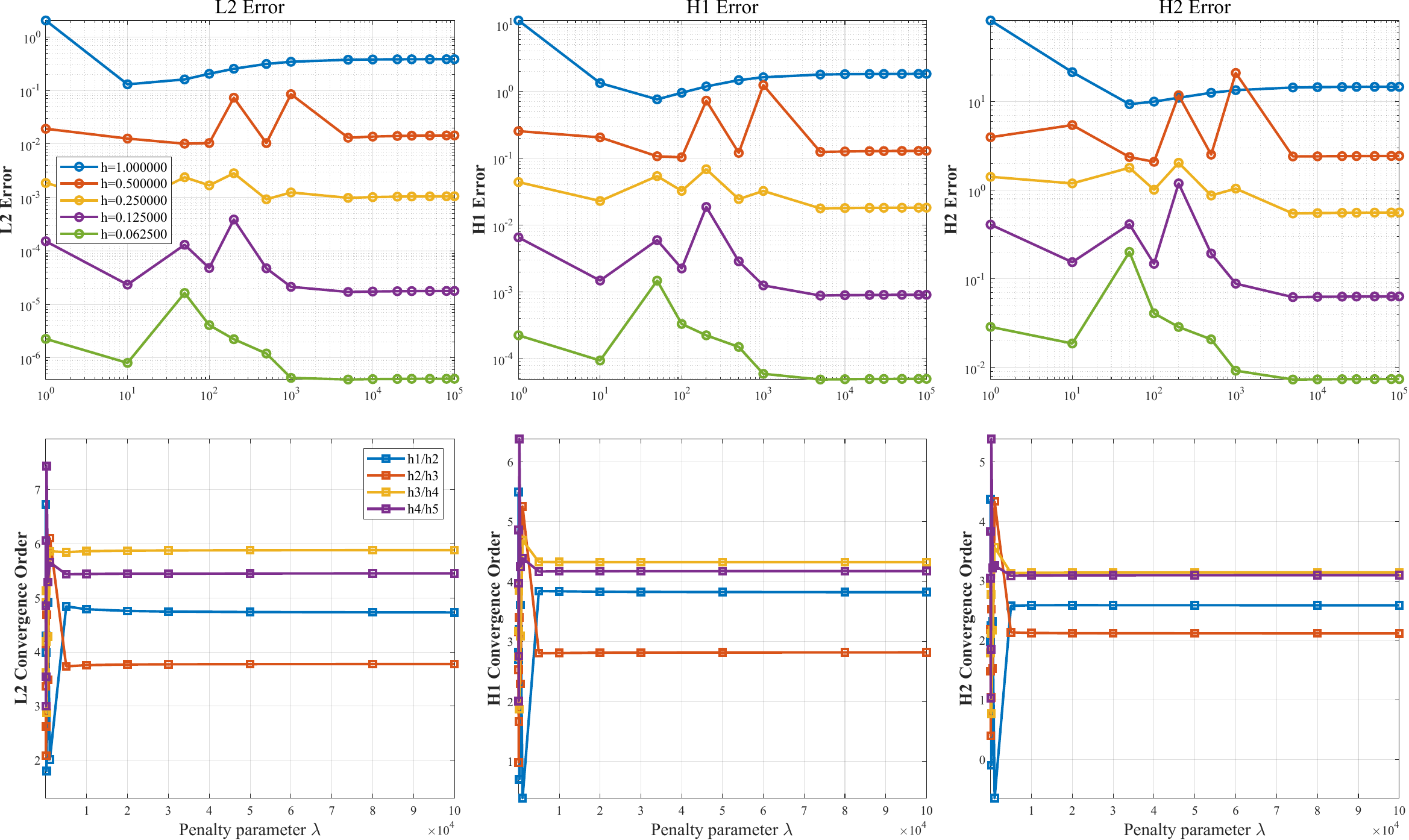}}
		\caption{Performance of the $P_4$ element on the square domain. Top row: Numerical errors versus penalty parameter $\lambda$. Bottom row: Convergence rates versus penalty parameter $\lambda$.}
		\label{fig:p4_s_combined}
	\end{figure}

	\begin{figure}[p]
		\centering
		\subfloat[Polynomial exact solution]{\includegraphics[width=1.0\textwidth]{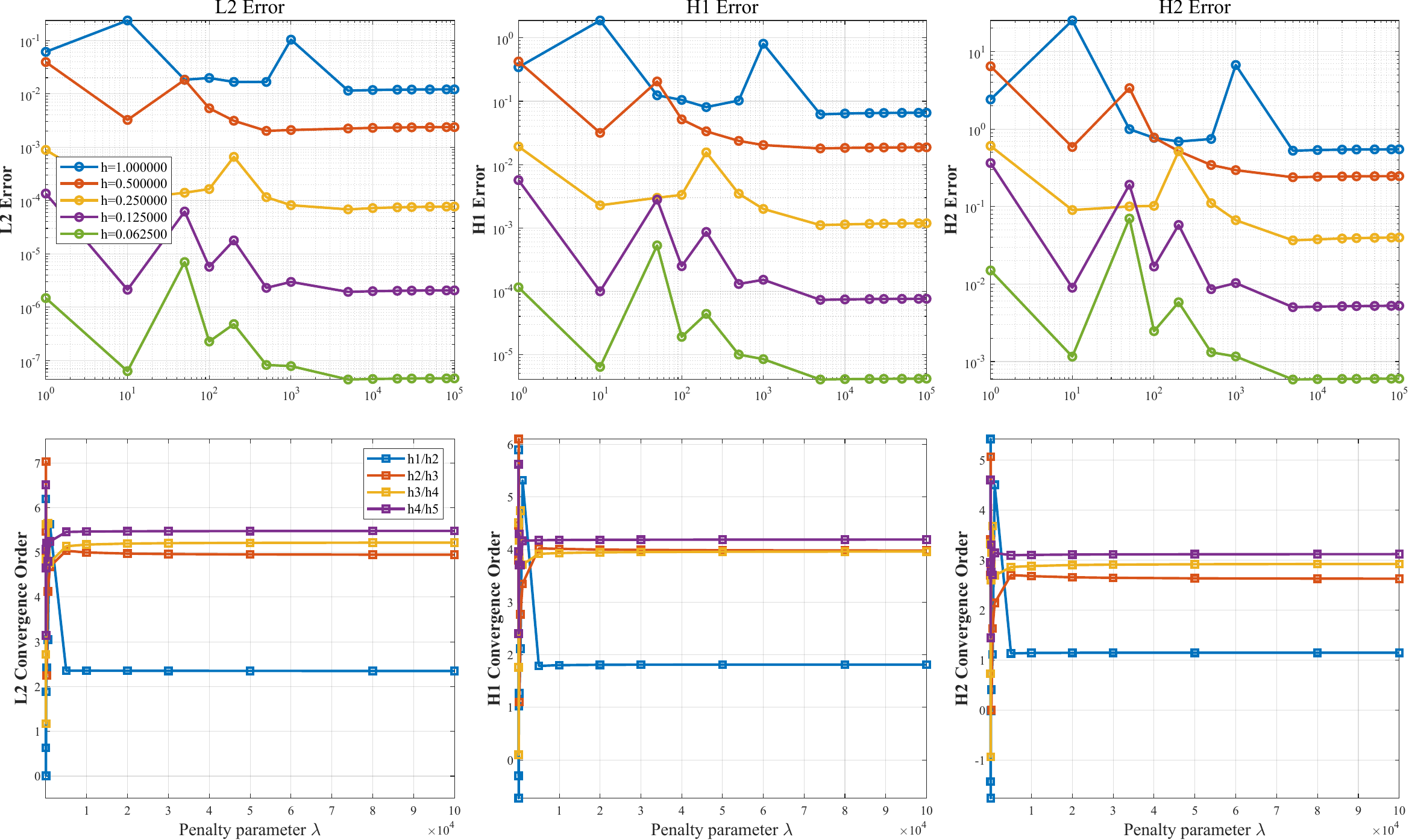}} \\
		\vspace{0.4cm}
		\subfloat[Trigonometric exact solution]{\includegraphics[width=1.0\textwidth]{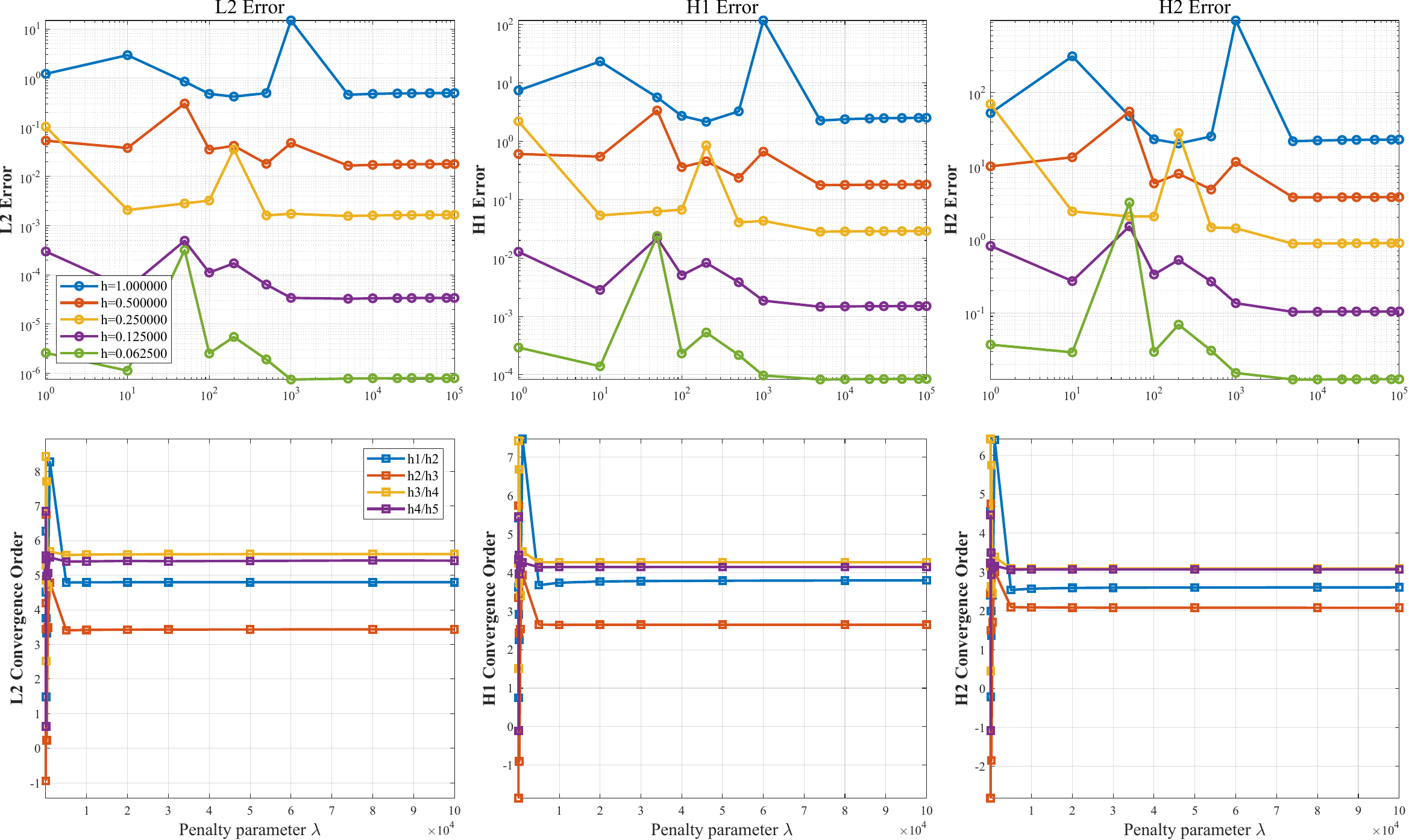}}
		\caption{Performance of the $P_4$ element on the L-shaped domain. Top row: Numerical errors versus penalty parameter $\lambda$. Bottom row: Convergence rates versus penalty parameter $\lambda$.}
		\label{fig:p4_l_combined}
	\end{figure}

	\section{Conclusion}
	\label{sec:conclusion}
	
	We have developed and rigorously analyzed a unified, high-order IPDG framework for fourth-order elliptic problems. By combining a partially continuous finite element space with degree-matched reduced-order projections, the method ensures that, as the penalty grows, the discrete solution converges to the solution of the nonconforming finite element method on $B_h^k$, the optimal nonconforming space. This intrinsic connection decouples the scheme's approximation capability from the penalty magnitude, thereby completely eliminating numerical locking for cubic and quartic elements.
	
	We emphasize the distinguishing feature of the present approach: in contrast to weakly over-penalized methods~\cite{Brenner2010WOPSIP}, which avoid locking by reducing the penalty on high-frequency modes, or mixed formulations that circumvent the issue by introducing auxiliary variables, our scheme achieves locking-free behavior within the pure IPDG framework \emph{uniformly for arbitrarily large} penalty parameters. 
	This robustness is a direct consequence of the fact that the constrained subspace $V_{h,\infty}^k$ coincides with the known optimal nonconforming space \(B_h^k\).
	
	Theoretical analysis proves optimal convergence rates of \(\mathcal{O}(h^{k-1})\) in the energy norm and \(\mathcal{O}(h^{k+1})\) in the \(L^2\) norm for \(k\in\{3,4\}\). Numerical experiments on both convex and non-convex geometries validate these findings, and extensive sensitivity analyses confirm that the errors and convergence rates remain uniformly stable for penalty parameters up to \(10^5\).
	
	Future work will aim at extending this framework to three-dimensional polyhedral meshes. We stress, however, that such an extension is highly nontrivial: in 3D the jumps and the corresponding projection operators must be defined on faces, and the construction of an exact discrete Stokes complex on general tetrahedral partitions is considerably more delicate for quartic elements. A full stability analysis in 3D remains a challenging open problem.

	\Acknowledgements{This paper is supported partially by the Strategic Priority Research Program of Chinese Academy of Sciences (Grant No. XDB0640000) and by the National Natural Science Foundation of China (Grant No. 12271512, No. 12371389). The authors thank Dr. Rui Ma for her constructive discussions, which were of great help to this article.}
	

	
	\appendix
	
	\addtocontents{toc}{\let\protect\numberline\protect\appendixnumberline}
	
	\section{Supplementary Test: Smooth Trigonometric Solution on the L-Shaped Domain}
	\label{sec:appendix_trig}
	We also test a trigonometric exact solution on the non-convex L-shaped domain:
	\begin{equation}
		u(x,y) = \sin(2\pi x)\sin(2\pi y).
		\label{eq:exact_trig_lshape}
	\end{equation}
	
	The computational results for the $P_3$ and $P_4$ elements under a fixed penalty parameter ($\lambda = 1000$) are reported in Table \ref{tab:trig_lshape_p3_convergence} and Table \ref{tab:trig_lshape_p4_convergence}. Similar to the polynomial case, the errors systematically decrease as the mesh size $h$ is halved. The energy norm errors for both elements consistently converge at the expected optimal rates (approximately $\mathcal{O}(h^2)$ for $P_3$ and $\mathcal{O}(h^3)$ for $P_4$). 
	\begin{table}[htbp]
		\footnotesize\tabcolsep 6pt
		\begin{center}
			\caption{Numerical errors and convergence rates of the $P_3$ IPDG method for the trigonometric solution on the L-shaped domain ($\lambda = 1000$).}
			\label{tab:trig_lshape_p3_convergence}
			\begin{tabular}{@{} c c c c c c c c @{}}
				\toprule
				\multirow{2}{*}{Level} & \multirow{2}{*}{$h$ (refinement level)} & \multicolumn{2}{c}{$L^2$-Norm} & \multicolumn{2}{c}{$H^1$-Norm} & \multicolumn{2}{c}{$H^2$-Norm} \\
				\cmidrule(lr){3-4} \cmidrule(lr){5-6} \cmidrule(l){7-8}
				& & Error & Order & Error & Order & Error & Order \\
				\midrule
				1 & 1.000000 & $4.8693 \times 10^{-1}$ &  --  & $2.4426 \times 10^{0}$ &  --  & $2.3540 \times 10^{1}$ &  --  \\
				2 & 0.500000 & $1.7846 \times 10^{-1}$ & 1.45 & $1.3544 \times 10^{0}$ & 0.85 & $1.3241 \times 10^{1}$ & 0.83 \\
				3 & 0.250000 & $8.4759 \times 10^{-3}$ & 4.40 & $1.3453 \times 10^{-1}$ & 3.33 & $3.5944 \times 10^{0}$ & 1.88 \\
				4 & 0.125000 & $4.7984 \times 10^{-4}$ & 4.14 & $1.5856 \times 10^{-2}$ & 3.08 & $9.1823 \times 10^{-1}$ & 1.97 \\
				5 & 0.062500 & $3.2448 \times 10^{-5}$ & 3.89 & $2.0578 \times 10^{-3}$ & 2.95 & $2.3398 \times 10^{-1}$ & 1.97 \\
				6 & 0.031250 & $2.4572 \times 10^{-6}$ & 3.72 & $2.7051 \times 10^{-4}$ & 2.93 & $5.9562 \times 10^{-2}$ & 1.97 \\
				\bottomrule
			\end{tabular}
		\end{center}
	\end{table}
	
	\begin{table}[htbp]
		\footnotesize\tabcolsep 6pt
		\begin{center}
			\caption{Numerical errors and convergence rates of the $P_4$ IPDG method for the trigonometric solution on the L-shaped domain ($\lambda = 1000$).}
			\label{tab:trig_lshape_p4_convergence}
			\begin{tabular}{@{} c c c c c c c c @{}}
				\toprule
				\multirow{2}{*}{Level} & \multirow{2}{*}{$h$ (refinement level)} & \multicolumn{2}{c}{$L^2$-Norm} & \multicolumn{2}{c}{$H^1$-Norm} & \multicolumn{2}{c}{$H^2$-Norm} \\
				\cmidrule(lr){3-4} \cmidrule(lr){5-6} \cmidrule(l){7-8}
				& & Error & Order & Error & Order & Error & Order \\
				\midrule
				1 & 1.000000 & $1.4870 \times 10^{1}$ &  --  & $1.1738 \times 10^{2}$ &  --  & $9.6426 \times 10^{2}$ &  --  \\
				2 & 0.500000 & $4.7918 \times 10^{-2}$ & 8.28 & $6.6274 \times 10^{-1}$ & 7.47 & $1.1444 \times 10^{1}$ & 6.40 \\
				3 & 0.250000 & $1.7419 \times 10^{-3}$ & 4.78 & $4.3355 \times 10^{-2}$ & 3.93 & $1.4263 \times 10^{0}$ & 3.00 \\
				4 & 0.125000 & $3.3903 \times 10^{-5}$ & 5.68 & $1.8607 \times 10^{-3}$ & 4.54 & $1.3613 \times 10^{-1}$ & 3.39 \\
				5 & 0.062500 & $7.4192 \times 10^{-7}$ & 5.51 & $9.7105 \times 10^{-5}$ & 4.26 & $1.5252 \times 10^{-2}$ & 3.16 \\
				6 & 0.031250 & $1.8952 \times 10^{-8}$ & 5.29 & $5.6846 \times 10^{-6}$ & 4.09 & $1.8240 \times 10^{-3}$ & 3.06 \\
				\bottomrule
			\end{tabular}
		\end{center}
	\end{table}
	
	\section{Sensitivity Analysis for the $P_3$ Element}
	\label{sec:appendix_sensitivity}
	
	The performance of the $P_3$ element on the regular square domain is depicted in Figure~\ref{fig:p3_s_combined}. The sensitivity analysis on the non-convex L-shaped domain shows the same stable behavior for large penalty parameters, as depicted in Figure~\ref{fig:p3_l_combined}.
			
	\begin{figure}[p]
		\centering
		\subfloat[Polynomial exact solution]{\includegraphics[width=1.0\textwidth]{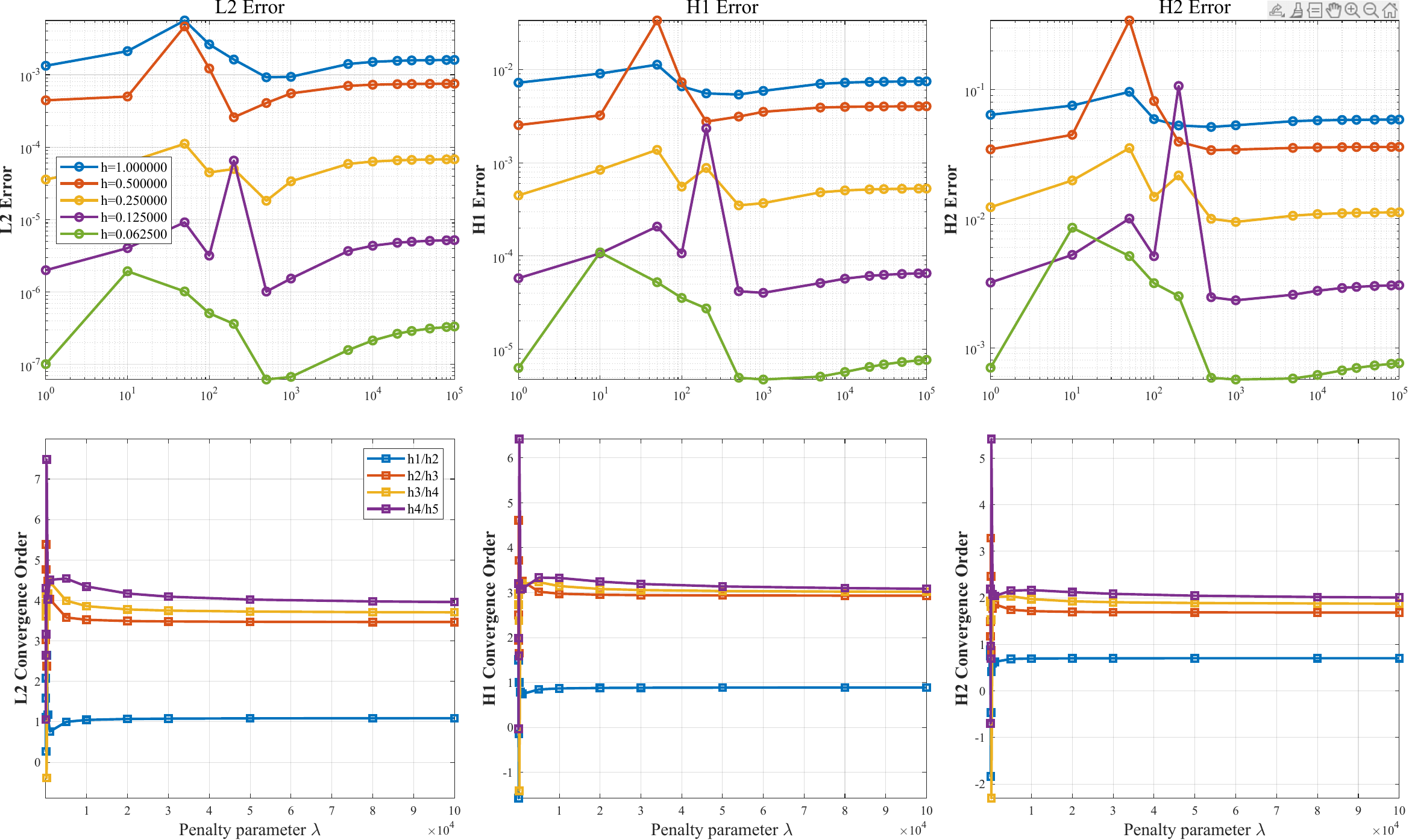}} \\
		\vspace{0.5cm}
		\subfloat[Trigonometric exact solution]{\includegraphics[width=1.0\textwidth]{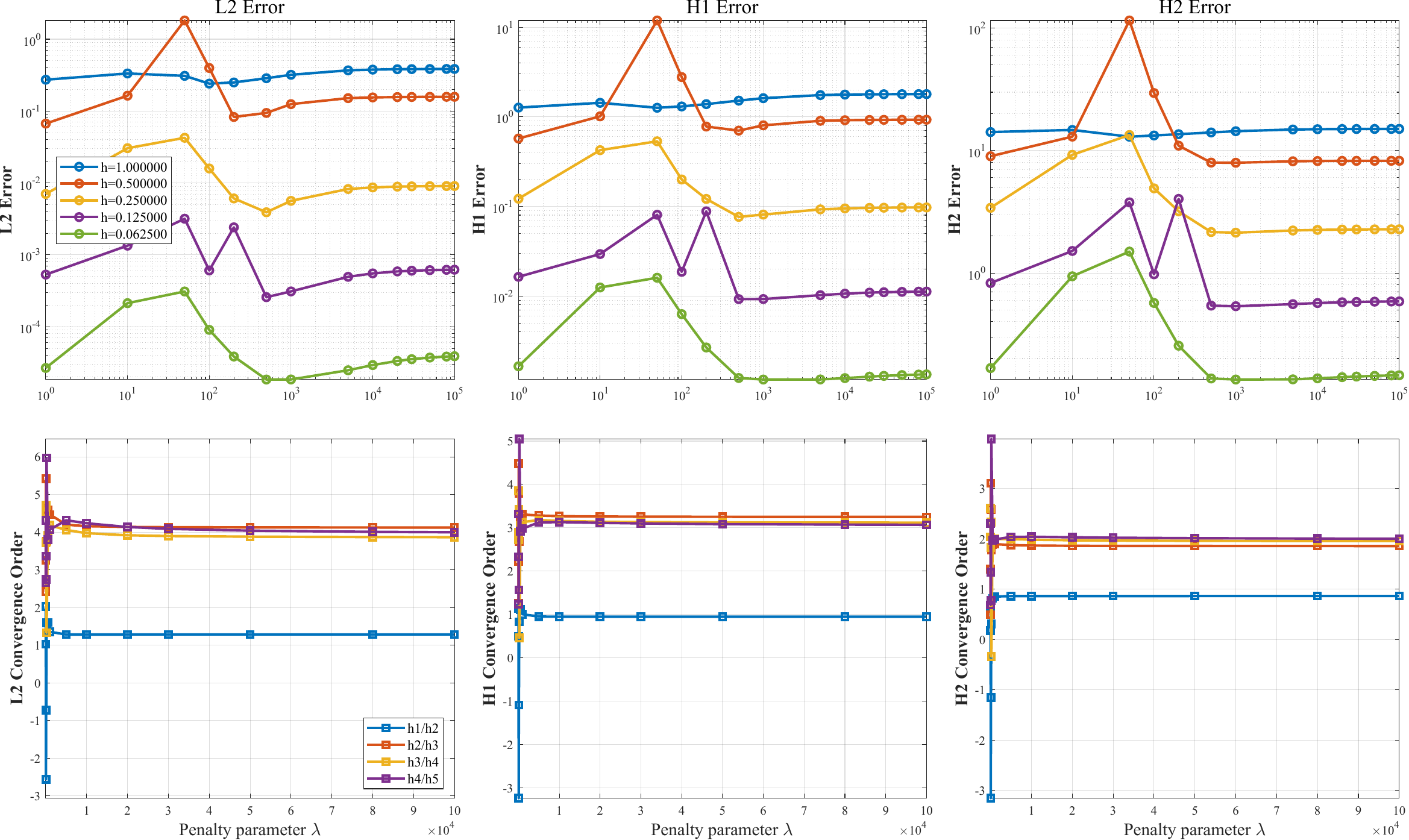}}
		\caption{Performance of the $P_3$ element on the square domain. Top row: Numerical errors versus penalty parameter $\lambda$. Bottom row: Convergence rates versus penalty parameter $\lambda$.}
		\label{fig:p3_s_combined}
	\end{figure}

	\begin{figure}[p]
		\centering
		\subfloat[Polynomial exact solution]{\includegraphics[width=1.0\textwidth]{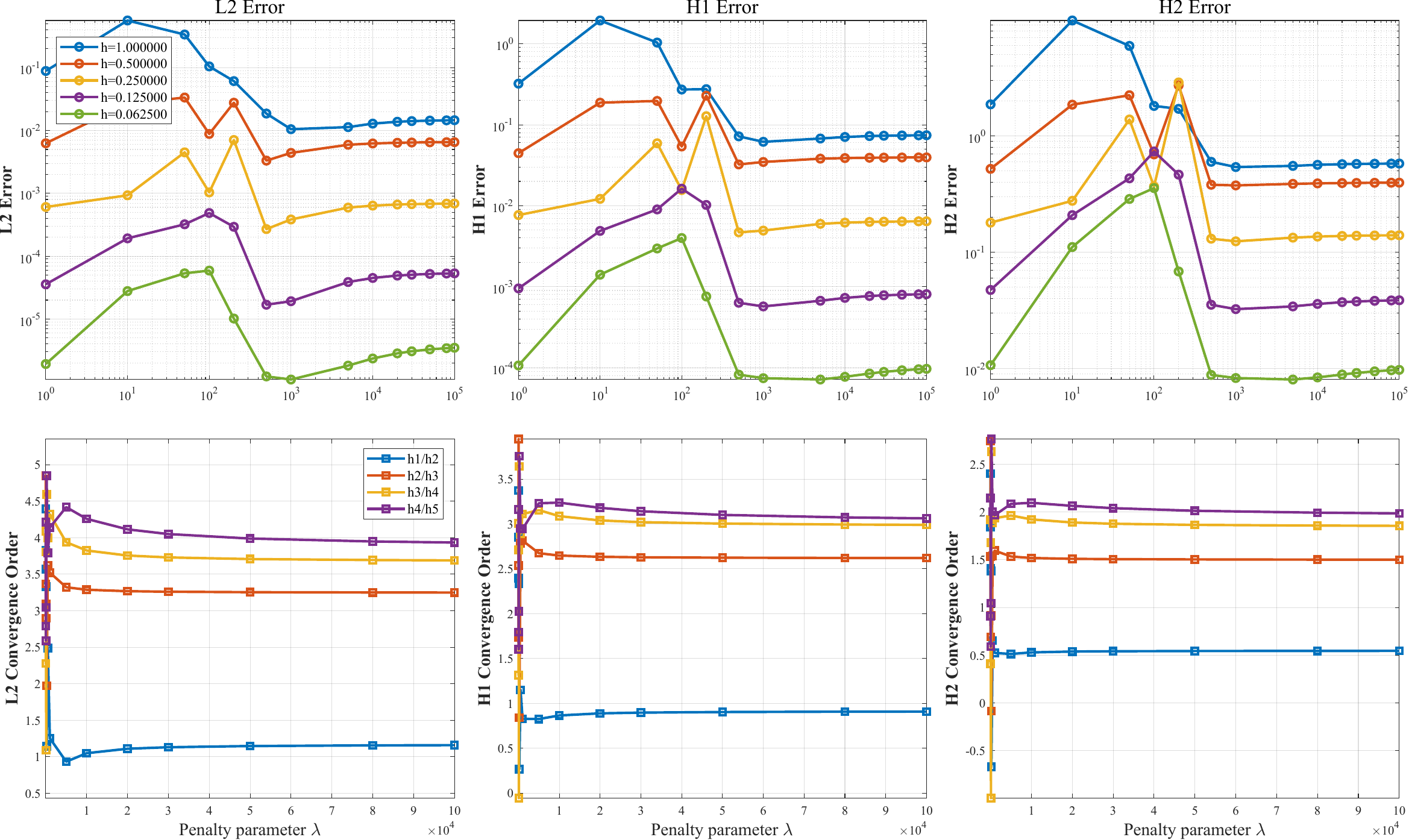}} \\
		\vspace{0.5cm}
		\subfloat[Trigonometric exact solution]{\includegraphics[width=1.0\textwidth]{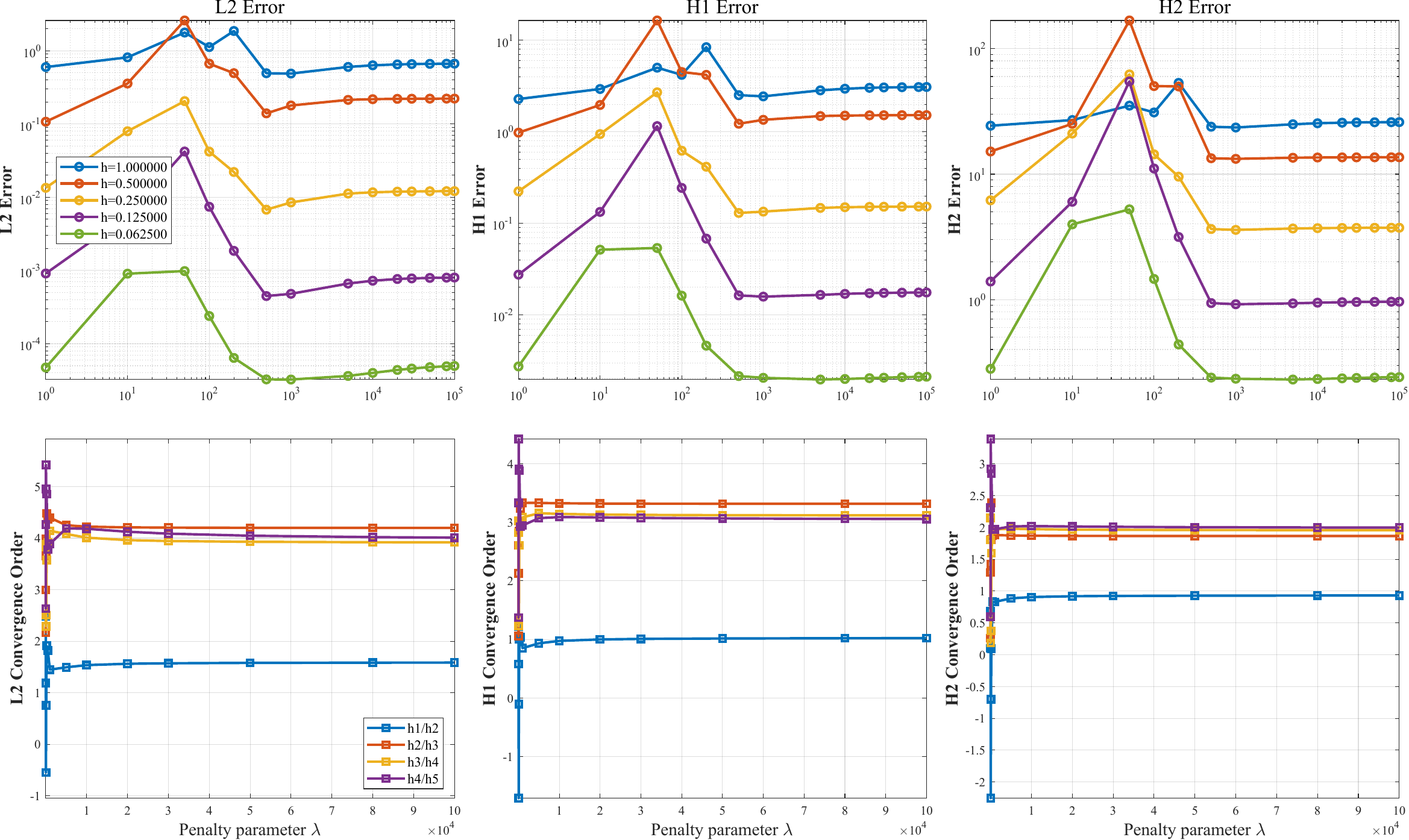}}
		\caption{Performance of the $P_3$ element on the L-shaped domain. Top row: Numerical errors versus penalty parameter $\lambda$. Bottom row: Convergence rates versus penalty parameter $\lambda$.}
		\label{fig:p3_l_combined}
	\end{figure}
\clearpage
\end{document}